\documentclass[11pt,leqno, a4paper]{amsart}
\usepackage{amsmath,amsthm,amscd,amssymb,amsfonts, amsbsy}
\usepackage{latexsym}
\usepackage{txfonts}
\usepackage{exscale}
\usepackage{graphicx}
\usepackage[colorlinks,citecolor=red,pagebackref,hypertexnames=false]{hyperref}
\usepackage{color}

\calclayout
\allowdisplaybreaks

\theoremstyle{plain}
\newtheorem{theorem}[equation]{Theorem}
\newtheorem{lemma}[equation]{Lemma}
\newtheorem{corollary}[equation]{Corollary}

\newtheorem{definition}[equation]{Definition}

\newtheorem{example}[equation]{Example}

\theoremstyle{remark}
\newtheorem{remark}[equation]{Remark}

\numberwithin{equation}{section}

\newcommand{\eps}{\varepsilon}

\newcommand{\dist}{\operatorname{dist}}

\newcommand{\re}{\mathbb{R}}
\newcommand{\rn}{\mathbb{R}^n}

\newcommand{\ree}{\mathbb{R}^{n+1}}
\newcommand{\N}{\mathbb{N}}

\newcommand{\C}{\mathcal{C}}

\newcommand{\om}{\Omega}

\newcommand{\nt}{n.t. \nabla\mathcal{S}}

\newcommand{\pom}{{\partial\Omega}}

\newcommand{\hm}{\omega}

\newcommand{\RNum}[1]{\uppercase\expandafter{\romannumeral #1\relax}}
\newcommand{\squeezeup}{\vspace{-10mm}}

\renewcommand{\emptyset}{\mbox{\textup{\O}}}

\DeclareMathOperator{\supp}{supp}

\DeclareMathOperator{\diam}{diam}

\DeclareMathOperator{\interior}{int}

\begin{document}
\allowdisplaybreaks

\title[Reifenberg Flatness and Oscillation of the Unit Normal Vector]{Reifenberg Flatness and Oscillation of the Unit Normal Vector}
\author{Simon Bortz}
\author{Max Engelstein}
\thanks{This work was mostly done while both authors were in residence at the MSRI Harmonic Analysis program, supported under NSF Grant No. DMS-1440140.  The first author was supported by the NSF INSPIRE Award DMS-1344235 and the second author was supported by an NSF MSPRF DMS-1703306.}
\subjclass[2010]{Primary 35R35, 49J52. Secondary 28A75, 31A15.}
\keywords{two-phase free boundary problems, harmonic measure, Reifenberg Flatness, Chord Arc Domains}
\address{	
	School of Mathematics, University of Minnesota, Minneapolis, MN
55455, USA}
	\email{bortz010@umn.edu} 
\address{Department of Mathematics\\Massachusetts Institute of Technology\\Cambridge, MA, 02139-4307 }
\email{maxe@mit.edu}

\begin{abstract}
We show (under mild topological assumptions) that small oscillation of the unit normal vector implies Reifenberg flatness. We then apply this observation to the study of chord-arc domains and to a quantitative version of a two-phase free boundary problem for harmonic measure previously studied by Kenig-Toro \cite{kenigtorotwophase}.
\end{abstract}

\maketitle

\tableofcontents

\section{Introduction}

The connections between the regularity of a domain (often expressed in terms of the oscillation of the unit normal)  and  potential theory, singular integrals and regularity for elliptic PDE has been a topic of considerable interest in mathematics (see, e.g. \cite{toronotices}).  An important object of study in this field are  chord arc domains. Roughly speaking, chord arc domains exhibit interior and exterior fatness, have quantitative connectivity and have (Ahlfors) regular surface measure. If there is sufficient control on the oscillation of the unit normal (in the $\mathrm{BMO}$ sense) and the domain is sufficiently flat (in the sense of Reifenberg \cite{Reifenberg}) we say the chord arc domain has small constant. Chord arc domains with small constant were introduced by Semmes in \cite{semmes1} and \cite{semmes2} and subsequently studied by Kenig and Toro (see e.g. \cite{kenigtoroannals}), Hofmann, Mitrea and Taylor (see e.g. \cite{hofmannmitreataylor}) and many others. In particular, chord arc domains are a natural setting in which to consider questions of regularity for harmonic functions or harmonic measure (see e.g. \cite{milakispiphertoro}, \cite{kenigtoroannals}). 

When Semmes introduced chord arc domains with small constant (see \cite{semmes1} and \cite{semmes2}) he made the assumption of {\it a priori} flatness and made the further restriction of working with $C^2$ surfaces. His focus was on operator theoretic and function theoretic properties of chord arc domains with small constant, e.g. Semmes showed the Cauchy integral operator restricted to a chord arc domain with small constant was ``almost" self adjoint. In addition, Semmes expressed interest in potential theoretic characterizations. These potential theoretic characterizations were investigated by Kenig and Toro, without the {\it a priori} assumption of smoothness but with the {\it a priori} assumption of Reifenberg flatness in \cite{kenigtoroduke}, \cite{kenigtoroannals} and \cite{kenigtoroannsci}. As a consequence of our main theorem (Theorem \ref{nuBMOthrm.thrm}), we show that the flatness hypothesis is redundant (see, e.g. Corollary \ref{BMOimpflatnessdomains.cor}), this in turn, should allow one to remove the {\it a priori} assumption of Reifenberg flatness from many theorems in the aforementioned works of Kenig and Toro (e.g.,  Theorem 4.2 in \cite{kenigtoroannals}). 

Our main theorem is essentially the following: under some mild assumptions on the topology (two sided corkscrews) and surface measure (Ahlfors regularity) of a domain, small oscillation of the unit normal implies flatness in the sense of Reifenberg \cite{Reifenberg}. In addition to the application mentioned above, we also use this observation to study a two-phase free boundary result first introduced by Kenig and Toro \cite{kenigtorotwophase} and examined further by the first author and Hofmann \cite{bortzhofmann}. Specifically, we can conclude a quantitative version of Theorem 1.1 in \cite{bortzhofmann} (see our Theorem \ref{FBthrm.thrm} below). 

Finally, we point out that there is a robust theory of Reifenberg-type parameterizations of surfaces whose unit normal has controlled oscillation (see, e.g. \cite{torobilip}, \cite{jessica1} and \cite{jessica2}).  Reifenberg-type parameterizations are a powerful tool in geometric analysis (see, e.g. \cite{nabervaltorta}) and we believe that there should be interesting connections between our work and these parameterizations. 

\subsection{Definitions}

Here we collect some definitions that we will need later on. The first is Ahlfors-regularity,

\begin{definition}[Ahlfors Regular (AR)]\label{def:AR} We say a closed set $E \subset \ree$ is Ahlfors regular (AR) if there exists a constant $C$ such that
\begin{equation}\label{ADR.eq}
C^{-1}r^n \le H^n(B(x,r)\cap E) \le Cr^n
\end{equation}
for all $x \in E$ and $r \in (0, \diam(E))$, where $H^n$ is the $n-$dimensional Hausdorff measure. 

When $\om$ is an open set we often write $\sigma \equiv H^n|_\pom$, the surface measure for $\om$. We may sometimes abuse terminology and say that $\sigma$ is an Ahlfors-regular measure, by which we mean that  \eqref{ADR.eq} holds with $E = \pom$. When referencing a dependence on the constant $C$ in \eqref{ADR.eq} we will simply write $AR$.
\end{definition}

Our second is corkscrew points. We need to guarantee that our domains are ``fat" on both the inside and out to prevent degeneracy. 

\begin{definition}[Two-sided Corkscrew Condition]\label{def:Corkscrews} We say an open set $\om \subset \ree$ satisfies the $(M,R_0)$ two-sided corkscrew condition if for every $x \in \pom$ and $r \in (0, R_0)$ there exist two balls $B_1 \equiv B(x_1,r/M)$ and $B_2 \equiv B(x_2, r/M)$ such that $B_1 \subset \om \cap B(x,r)$ and $B_2 \subset \om^c \cap B(x,r)$, where $\om^c$ denotes the compliment of $\om$. We call $x_1$ and $x_2$ interior and exterior corkscrew points respectively.

\end{definition}

There are several connections between Ahlfors regularity and corkscrews; if a domain, $\om$, satisfies a two-sided corkscrew condition then it is automatically lower Ahlfors regular. Moreover, David and Jerison \cite{davidandjerison} observed that if $\pom$ is Ahlfors regular and $\om$ satisfies the two-sided corkscrew condition then $\pom$ is uniformly rectifiable (see Definition \ref{defur} below). Our next condition is a quantitative measure of connectedness. 

\begin{definition}[Harnack Chain Condition]\label{def:HarnackChains} Following \cite{jerisonandkenig}, we say
that $\Omega$ satisfies the $(C, R)$-Harnack Chain condition if
for every $0 < \rho \leq R,\, \Lambda\geq 1$, and every pair of points
$X,X' \in \Omega$ with $\delta(X),\,\delta(X') \geq\rho$ and $|X-X'|<\Lambda\,\rho$, there is a chain of
open balls
$B_1,\dots,B_N \subset \Omega$, $N\leq C\log_2\Lambda + 1$,
with $X\in B_1,\, X'\in B_N,$ $B_k\cap B_{k+1}\neq \emptyset$
and $C^{-1}\diam (B_k) \leq \dist (B_k,\partial\Omega)\leq C\diam (B_k).$  The chain of balls is called
a ``Harnack Chain''.
\end{definition}

Domains which are both quantitatively fat and quantitatively connected are called NTA (Non-tangentially accessible), and were introduced by Jerison and Kenig \cite{jerisonandkenig} as a natural setting in which the boundary behavior of harmonic functions can be understood. 

\begin{definition}[NTA and Chord Arc Domains]\label{def:CAD} We say $\om \subset \ree$ is a Non-Tangentially Accessible Domain (NTA) with constants $(M, R_0)$, if it satisfies the $(M, R_0)$- Harnack chain condition and the $(M, R_0)$ two-sided corkscrew condition. If $\Omega$ is unbounded, we require that $\mathbb R^n \backslash \partial \Omega$ consists of two, non-empty, connected components. Note that if $\Omega$ is unbounded, then $R_0 = \infty$ is allowed. 

Finally, if $\om$ is an NTA domain whose boundary is Ahlfors regular we say $\om$ is a chord arc domain.
\end{definition}

In the definition of unbounded NTA domains, it is sometimes required that $R_0 = \infty$ (see, e.g. \cite{kenigtoroduke}, \cite{kenigtorotwophase}). This is to obtain estimates on harmonic measure/functions at arbitrarily large scales. Since we are only interested in the geometric properties of $\Omega$, we allow $R_0 < \infty$ even for unbounded $\Omega$. For further discussion of unbounded domains, see Remark \ref{unboundedspellstrouble} below.

We need to also measure the ``flatness" of a set.  Let $E \subset \ree$ be a locally compact set. For $Q \in E$ and $r>0$ define
$$\Theta(Q,r) = \inf_L \left\{\frac{1}{r}D( E \cap B(Q, r),L(Q, r) \cap B(Q, r))  \right\}$$
where the infimum is taken over all $n-$planes containing $Q$. Here D denotes the Hausdorff distance, that is, for $A,B \subset \ree$,
$D[A,B] = \sup\{d(a,B): a \in A\}  + \sup\{d(b,A): b \in B\}$. With this in hand, we can define flatness after Reifenberg \cite{Reifenberg};

\begin{definition}[Reifenberg Flat and Vanishing Reifenberg Flat]\label{def:Reifenberg}
We say $E$ is 
$\delta-$Reifenberg
 flat for some $\delta > 0$ if for each compact set $K \subset \ree$ there exists $R_K$ such that
$$\sup_{r \in (0,R_K]}\sup_{Q \in K \cap E} \Theta(Q,r) < \delta.$$
We say $E$ is $(\delta, R)-$Reifenberg flat if 
$$\sup_{r \in (0,R]}\sup_{Q \in E} \Theta(Q,r) < \delta.$$ 
We say $E$ is vanishing Reifenberg flat if for every compact set $K \subset \ree$
$$\lim_{r \to 0} \sup_{Q \in E \cap K} \Theta(Q,r) = 0.$$
%We shall throughout assume that $\delta < \delta_n$, where $\delta_n$ is chosen appropriately small, it is easy to see that any of the definitions of Reifenberg flatness are only meaningful when $\delta$ is small.

Finally, we say that a domain $\Omega \subset \mathbb R^{n+1}$ is $\delta$-Reifenberg flat (or $(\delta, R)$-Reifenberg flat, vanishing Reifenberg flat), if $\partial \Omega$ is $\delta$-Reifenberg flat (resp. $(\delta, R)$-Reifenberg flat, vanishing Reifenberg flat) and $\Omega$ satisfies the {\bf separation condition}: for every compact set $K \subset \ree$ there exists $R > 0 $ such that for $Q \in \pom \cap K$ and $r \in (0,R]$ there exists
an n-dimensional plane $L(Q, r)$ containing $Q$ and a choice of unit normal vector to $L(Q, r)$,
$\vec{n}_{Q,r}$, satisfying
$$T^+(Q, r) = X = (x, t) = \left\{x +t
\vec{n}_{Q,r} \in B(Q, r): x \in L(Q, r), t > \tfrac{r}{4} \right\} \subset \om, $$
and
$$T^-(Q, r) = X = (x, t) = \left\{x +t
\vec{n}_{Q,r} \in B(Q, r): x \in L(Q, r), t < \tfrac{-r}{4} \right\} \subset \om^c.$$
Additionally, if $\om$ is unbounded we have the further requirement that $\ree \setminus \pom$ consists of two connected components $\om$, that $\interior(\om^c) \neq \emptyset$ and that $\partial \Omega$ is $(\delta_n, R)$-Reifenberg flat for some $R > 0$. Here $\delta_n > 0$ is chosen small enough so that $\Omega$ is an NTA domain (up to scale $R_0 = R/10$, see Lemma 3.1 in \cite{kenigtoroduke}). 
\end{definition}

For unbounded domains, it is often usual to assume that $\partial \Omega$ is $(\delta_n, \infty)$-Reifenberg flat. This is to ensure that $\Omega$ is an NTA domain at scale $R_0 = \infty$. Since we allow unbounded NTA domains to have local estimates, we only require that unbounded Reifenberg flat domains have local flatness. Again see Remark \ref{unboundedspellstrouble} below.

To simplify future proofs, let us make a quick remark on how the separation condition interacts with the two-sided corkscrew condition.

\begin{remark}\label{separationredundant}
For a bounded domain $\Omega$, we note that if $\partial \Omega$ is $(\delta, R_1)$-Reifenberg flat (in the sense of sets) and $\Omega$ satisfies the $(M, R_0)$-two-sided corkscrew condition, then there exists a $\delta_0 \equiv \delta_0(M, R_0, R_1) > 0$ such that if $\delta < \delta_0$, then $\Omega$ is a $(\delta, R_2)$-Reifenberg flat domain, where $R_2 = \min(R_0, R_1)/2$. The same holds for unbounded $\Omega$ with the additional {\it a priori} assumption that $\mathbb R^{n+1} \setminus \pom$ consists of two connected, non-empty components, one of which is $\Omega$. 

To see this, note that if $\delta$ is small enough (compared to $M$) and $L(Q, r)$ is the plane which best approximates $B(Q,r)\cap \pom$, then both the interior and exterior corkscrew points to $Q$ at scale $r$, call them $A^{\pm}$, are not contained in the $\delta r$-slab around $L(Q,r)$, call it $S_\delta(Q,r)$. Furthermore they must be on different sides of $L(Q,r)$ (otherwise the segment between them will lie outside $S_\delta(Q,r)$ but contain a point of $\pom$, contradicting the Reifenberg flatness of $\pom$). Similarly, if there is a point, $y$, in $\Omega^c\cap B(Q,r)$ which is outside the $r/4$-slab of $L(Q,r)$ on the same side of $L(Q,r)$ as $A^+$ then the segment between $A^+$ and $y$ must contain a point in $\pom$ outside of $S_\delta(Q,r)$, again a contradiction of Reifenberg flatness. As a similar argument holds for points in $\Omega\cap B(Q,r)$ on the same side of $L(Q,r)$ as $A^-$, the  separation property follows.
\end{remark}
%\ME{Do you feel like this proof above is necessary?}

Finally, we need a measure of control on the oscillation of a function. We introduce the classic BMO and VMO function spaces (note that if $\pom$ is Ahlfors regular, then it is a space of homogeneous type, see, e.g. \cite{SHT}, and thus much of the classical theory of these spaces extends to Ahlfors regular domains). 

\begin{definition}[BMO and VMO] Let $\om \subset \ree$ be a set of locally finite perimeter with $\pom$ Ahlfors regular. Let $f \in L^2_{loc}(d\sigma)$ where $\sigma = H^n|_\pom$, we say that $f \in BMO(d\sigma)$ if
$$\lVert f \rVert_{BMO(d\sigma)} = \sup_{s > 0}\sup_{Q \in \pom} \left( \fint_{B(Q,s)}\left| f(z) - \fint_{B(Q,s)} f(x) \, d\sigma(x) \right|^2 \, d\sigma(z) \right)^{\frac{1}{2}}< \infty.$$
We denote by $\mathrm{VMO}$ the closure of uniformly continuous functions on $\pom$ in the $BMO$-norm. There is also a notion of $\mathrm{VMO}_{loc}$; $f\in \mathrm{VMO}_{loc}$ if for every compact set $K \subset \ree$
$$\lim_{s \to 0}\sup_{Q \in \pom \cap K} \left( \fint_{B(Q,s)}\left| f(z) - \fint_{B(Q,s)} f(x) \, d\sigma(x) \right|^2 \, d\sigma(z) \right)^{\frac{1}{2}} = 0.$$
\end{definition}

\begin{remark}\label{vmolocvsvmo}
We should remark that what we call $\mathrm{VMO}_{loc}$ (after \cite{bortzhofmann}) is actually called $\mathrm{VMO}$ in some points in the literature (see, e.g. \cite{kenigtoroannals}). However, the definition given above is more suited to unbounded domains; in particular, for bounded domains the two definitions are equivalent, but for unbounded domains the ``closure" definition controls the behavior of the function at large scales (see the discussion in \cite{kenigtoroannsci}). 
\end{remark}

\section{Small $\mathrm{BMO}$ norm implies Reifenberg flat}\label{sec:mainarg}

The goal of this section is to prove our main theorem, Theorem \ref{nuBMOthrm.thrm}. 

\begin{theorem}\label{nuBMOthrm.thrm} Suppose $\om \subset \ree$ is an open set satisfying the $(M,R_0)$-two-sided corkscrew condition whose boundary is Ahlfors regular. Then for every $\delta > 0$ there exists $c_\delta$ and $r_\delta$ depending only on $\delta$, the Ahlfors regularity constant and the constants in the two-sided corkscrew condition such that if 
$$\lVert \nu \rVert_{BMO(d\sigma)} < c_\delta,$$
(where $\sigma$ is the surface measure for $\om$ and $\nu$ is the outer unit normal to $\om$), then $\pom$ is $(\delta, r_\delta)-$Riefenberg flat. In light of Remark \ref{separationredundant}, we can choose $c_\delta$ small enough such that  $\om$ also satisfies the separation condition and thus is a  $\delta$-Reifenberg flat domain.
\end{theorem}

\begin{remark}\label{relationshiptoHMT}
We are grateful to Steve Hofmann for pointing out Theorem 4.19 in \cite{hofmannmitreataylor}, which is very closely related to our Theorem \ref{nuBMOthrm.thrm}. There it is shown (as a corollary of a much larger theory) that control on the oscillation of the unit normal implies Reifenberg flatness, under the {\it a priori} assumption that $\Omega$ is a two-sided John domain (see \cite{hofmannmitreataylor} for more details and definitions). 

Our result is more general as the two-sided John condition is replaced by the larger class of two-sided corkscrew domains. We also remark that the methods of proof are completely different; in \cite{hofmannmitreataylor} the John condition is used to establish a Poincar\'e inequality (see Proposition 4.13 there) on the boundary which in turn allows for a Semmes-type decomposition (Theorem 4.16 there). Reifenberg flatness then follows easily. In contrast, our proof is by compactness and uses only elementary real analysis estimates (along with some machinery from the theory of domains with locally finite perimeter). 

Finally, we remark that in light of Theorem \ref{nuBMOthrm.thrm} and some of the subsequent corollaries,  the assumption of two-sided John can be replaced by two-sided corkscrew in some of the theorems in \cite{hofmannmitreataylor} (e.g. in their Theorem 4.16). 
\end{remark}

We begin the proof of Theorem \ref{nuBMOthrm.thrm} with a compactness lemma--the proof of which is essentially contained in \cite{kenigtoroannsci}.

\begin{lemma}[{\cite[Theorem 4.1]{kenigtoroannsci}}, {\cite[Appendix B]{badgerengelsteintoro}}]\label{KTlemma} Let $M, R_0 > 0$. Suppose for each $i \in \N$, $\om^{(i)}\subset \ree$ is an open set with uniformly $n-$Ahlfors-regular boundary, which satisfies the $(M,R_0)$-two sided corkscrew condition and $0 \in \pom^{(i)}$. Let $0 < r_i < \infty$ with $r_i \downarrow 0$ and set
$$\om_i^{+} := \frac{1}{r_i}\om^{(i)}, \quad \om_i^{-} := \interior((\om_i^+)^c), \quad \pom_i := \frac{1}{r_i}\pom^{(i)}.$$
Then there exists a subsequence (which may depend on the $\{r_i\}$), which we relabel such that 
\begin{enumerate}
\item $\om_i^+ \to \om_\infty^+$ in the Hausdorff distance uniformly on compact sets,
\item $\om_i^- \to \om_\infty^- = \interior( (\om_\infty^+)^c) $  in the Hausdorff distance uniformly on compact sets and
\item $\pom_i \to \pom_\infty^+ = \pom_\infty^-$  in the Hausdorff distance uniformly on compact sets.
\end{enumerate}
Moreover, $\om_\infty^+$ (and $\om_\infty^-$) satisfies the two-sided corkscrew condition at all scales, is a set of locally finite perimeter whose topological boundary coincides with its measure theoretic boundary and $\pom_\infty^+$ is Ahlfors-regular. Finally, the corkscrew and Ahlfors-regularity constants of $\om_\infty$ depend only on the ambient dimension and the corkscrew and Ahlfors-regularity constants of the $\om_i$. 
\end{lemma}
\begin{proof} 
That domains which satisfy a two-sided $(M, R_0)$-corkscrew condition are closed in the Hausdorff distance sense follows from arguing as in \cite{kenigtoroannsci}. For a proof with full details see \cite{badgerengelsteintoro} Appendix B (the proof there shows that two-sided NTA domains form a closed class, but truncating the proof after Step 5/2 gives a complete argument for two-sided corkscrew domains). Elementary change of variables shows that if $\Omega$ is a $(M, R_0)$-corkscrew domain, then $\Omega/\rho$ is a $(M, R_0/\rho)$-corkscrew domain. Since $r_i \downarrow 0$ it follows that for any $R_1$, arbitrarily large, that $\Omega_i$ is a two-sided $(M, R_1)$-corkscrew domain (for large $i$ depending on $R_1, R_0$). Therefore, $\Omega^\pm_\infty$ satisfies the corkscrew condition at all scales with constant $M$. 

The remaining conclusions follow exactly as in \cite{kenigtoroannsci}, Theorem 4.1. Briefly, that $\Omega$ is a set of locally finite perimeter follows from the compactness of BV functions (the corkscrew condition guarantees that $\chi_{\om_i}\rightarrow \chi_{\om_\infty}$ in $L^1_{\mathrm{loc}}$). That the topological boundary is the same as the measure theoretic boundary follows from the existence of corkscrews. The upper Ahlfors regularity follows from the lower-semi continuity of the surface measure and the lower Ahlfors regularity follows from the existence of corkscrew points. 
\end{proof}

With this lemma in hand, we can prove Theorem \ref{nuBMOthrm.thrm} by means of a compactness argument. Before continuing, we point out the primary difficulty is the lack of a ``Portmanteau theorem" for signed measures. Recall that if $\mu_i$ are positive measures with $\mu_i \rightharpoonup \mu$ and $\mu(\partial A) = 0$, then $\mu_i(A) \to \mu(A)$, but this is not true for signed measures (take $\mu_i = \delta_{1-1/n} - \delta_1$ on the real line and $A = (0,1)$ and note that the $\mu_i$ converge weakly to the zero measure). 

\begin{proof}[Proof of Theorem \ref{nuBMOthrm.thrm}]

We proceed by contradiction. If the theorem is false, then there is a $\delta>0 $ and a sequence of domains $\om^{(i)}$ with uniform control on the Ahlfors regularity and corkscrew constants such that $\lVert \tilde{\nu}_i \rVert_{BMO(d\tilde{\sigma}_i)} < 1/i$, (where $\tilde{\sigma}_i$ is the surface measure for $\om^{(i)}$ and $\tilde{\nu}_i$ is the outer unit normal to $\om^{(i)}$) but $\pom_i$ fails to be $(\delta, 1/i)-$Reifenberg flat. In particular, there are points $Q_i \in \pom^{(i)}$ and $r_i < 1/i$ such that $\Theta_{\om^{(i)}}(Q_i, r_i) \ge \delta$ for some $r_i < 1/i$. After a harmless translation, we assume that $Q_i \equiv 0$.

Apply Lemma \ref{KTlemma} to $\om_i := \tfrac{1}{r_i}\om^{(i)}$. Let $\sigma_i$ and $\nu_i$ be the surface measure and unit outer normal for $\om_i$ respectively. Similarly, let $\sigma_\infty$ and $\nu_\infty$ be the surface measure and unit outer normal for $\om_\infty$, the limiting domain obtained from Lemma \ref{KTlemma}.  Recall that our assumptions imply that $\Theta_{\om_{i}}(0, 1) > \delta$, which implies, through the triangle inequality, that $\Theta_{\om_\infty}(0, 1/2) > \delta/2$. We will show, in fact, that $\om_\infty$ is a half-space, which will provide the desired contradiction. 

Note that $0 \in \pom_i$ for all $i$ and 
$$\left| \fint_{B(0,1)} \nu_i \, d\sigma_i \right| \le \fint_{B(0,1)} |\nu_i| \, d\sigma_i = 1.$$
Therefore, there exists a subsequence, which we relabel and fix henceforth, such that
$$ \fint_{B(0,1)} \nu_i \, d\sigma_i \to \vec{N},$$
for some vector $\vec{N}$, with $|\vec{N}| \leq 1$. Our goal is to show that $\nu_\infty(P) \equiv \vec{N}$ for $d\sigma_\infty$-almost every point $P \in \partial \Omega_\infty$. Once we have done this it follows that $\Omega_\infty = \vec{N}^\perp$ and we have reached the desired contradiction (to see these details, follow the proof in \cite{evansandgariepy} page 202).

The first step is to show that $|\vec{N}| = 1$. Using that $\lVert \tilde{\nu_i} \rVert_{BMO(d\tilde{\sigma}_i)} \to 0$ and the triangle inequality we have 
\begin{equation}
\begin{split}
0 &= \limsup_{i \to \infty} \fint_{B(0, r_i)} \left|\tilde{\nu}_i(z)- \fint_{B(0,r_i)} \tilde{\nu}_i(x) \, d\tilde{\sigma}_i(x)\right| \, d\tilde{\sigma}_i(z)
\\ &= \limsup_{i \to \infty} \fint_{B(0, 1)} \left|{\nu}_i(z) -\fint_{B(0,1)} {\nu}_i(x) \, d{\sigma}_i(x)\right| \, d{\sigma}_i(z)
\\ &\ge \limsup_{i \to \infty} \fint_{B(0, 1)} \left|1 - \left|\fint_{B(0,1)} {\nu}_i(x) \, d{\sigma}_i(x)\right|\right| \, d{\sigma}_i(z)
\\ & = \lim_{i \to \infty} \left| 1 - \left| \fint_{B(0,1)} \nu_i \, d\sigma_i\right| \right| = |1 - |\vec{N}||.
\end{split}
\end{equation}
Our next step is to prove the equivalence of limits of averages on the boundary; let $R > \rho > 0$, then by the uniform Ahlfors-regularity of the $\pom_i$, 
\begin{equation}\label{limat0.eq}
\begin{split}
&\limsup_{i\to \infty}\left| \fint_{B(0,\rho)} \nu_i(z) \, d\sigma_i(z) - \fint_{B(0,R)} \nu_i(x) \, d\sigma_i(x)\right|
\\& \quad \le \limsup_{i \to \infty} \fint_{B(0,\rho)} \left| \nu_i(z) - \fint_{B(0,R)} \nu_i(x) \, d\sigma(x) \right| \, d\sigma_i(z)
\\& \quad \le  \limsup_{i \to \infty}C\left(\frac{R}{\rho}\right)^{n}  \fint_{B(0,R)} \left| \nu_i(z) - \fint_{B(0,R)} \nu_i(x) \, d\sigma(x) \right| \, d\sigma_i(z)
\\& \quad \le \limsup_{i \to \infty}C\left(\frac{R}{\rho}\right)^{n}  \fint_{B(0,Rr_i)} \left| \tilde{\nu}_i(z) - \fint_{B(0,Rr_i)} \tilde{\nu}_i(x) \, d\tilde{\sigma}(x) \right| \, d\tilde{\sigma}_i(z) = 0
\end{split}
\end{equation}
where we used $\lVert \tilde{\nu_i} \rVert_{BMO(d\tilde{\sigma}_i)} \to 0$ in the last line. Similarly, if $P \in \pom_\infty$ and $\rho > 0$, then, setting $R = 10(\rho + |P|)$, we have 
\begin{equation}\label{limatp.eq}
\begin{split}
&\limsup_{i\to \infty}\left| \fint_{B(P,\rho)} \nu_i(z) \, d\sigma_i(z) - \fint_{B(0,R)} \nu_i(x)\, d\sigma_i(x)\right|
\\& \quad \le \limsup_{i \to \infty} \fint_{B(P,\rho)} \left| \nu_i(z) - \fint_{B(0,R)} \nu_i(x) \, d\sigma(x) \right| \, d\sigma_i(z)
\\& \quad \le  \limsup_{i \to \infty}C\left(\frac{R}{\rho}\right)^{n}  \fint_{B(0,R)} \left| \nu_i(z) - \fint_{B(0,R)} \nu_i(x) \, d\sigma(x) \right| \, d\sigma_i(z)
\\& \quad \le \limsup_{i \to \infty}C\left(\frac{R}{\rho}\right)^{n}  \fint_{B(0,Rr_i)} \left| \tilde{\nu}_i(z) - \fint_{B(0,Rr_i)} \tilde{\nu}_i(x) \, d\tilde{\sigma}(x) \right| \, d\tilde{\sigma}_i(z) = 0.
\end{split}
\end{equation}
Using \eqref{limat0.eq}, \eqref{limatp.eq} and
$$\lim_{i \to \infty} \fint_{B(0,1)} \nu_i \, d\sigma_i = \vec{N}$$
we have established that
\begin{equation}\label{alllimeq.eq}
\lim_{i \to \infty} \fint_{B(P,R)} \nu_i \, d\sigma_i = \lim_{i \to \infty} \fint_{B(0,1)} \nu_i \, d\sigma_i = \vec{N},
\end{equation}
for all $P \in \pom_\infty$ and $R > 0$. 

We are now ready to prove that $\nu_\infty(P) = \vec{N}$ for every $P \in \pom_\infty^*$ (where,  $\pom_\infty^*$ is the reduced boundary of $\om_\infty$, see \cite[Chapter 5]{evansandgariepy}). We collect several facts abound sets of locally finite perimeter many of which can be found in \cite[Chapter 5]{evansandgariepy}.  Recall, that if $P\in \partial^*\om_\infty$, then 
$$\lim_{R \downarrow 0^+} \fint_{B(P,R)} \nu_\infty \, d\sigma_\infty = \nu_\infty(P)$$
and $|\nu_\infty(P)| = 1$. Also $\sigma_\infty(\partial_*\Omega_\infty \setminus \partial^*\om_\infty) = 0$, where $\partial_*\om_\infty$ is the measure theoretic boundary. Since our domain has corkscrew points, $\partial_*\om_\infty = \pom_\infty$ and $\sigma_\infty$-almost every $P \in \partial \Omega$ is in $\partial^*\om_\infty$.

Let $0 < \epsilon < 1$ and $P \in \partial^*\om_\infty$. Choose $R > 0$ so that $\left|\nu_\infty(P) - \fint_{B(P,R)}\nu_\infty(Q) d\sigma_\infty(Q)\right| < \epsilon$ and $\sigma_\infty(\partial B(P,R)) = 0$. Such an $R$ exists by the definition of $\partial^*\om_\infty$ and the fact that $\sigma_\infty(\partial B(P,R)) = 0$ for all but countably many $R > 0$ (this follows from the Ahlfors regularity of $\sigma_\infty$). Let $\varphi \equiv \vec{N}\xi$ where $\xi \in C_c(\ree)$ with $\chi_{B(P,R)} \le \xi \le \chi_{B(P,2R)}$ and $\lVert \xi - \chi_{B(P,R)}\rVert_{L^2(d\sigma_\infty)} < \epsilon\sigma_\infty(B(P,R))^{1/2}$ (the existence of $\xi$, whenever $\sigma_\infty(\partial B(P, R)) =0$, follows from the continuity of $\sigma_\infty$ at $B(P,R)$ and Urysohn's lemma).

We establish some simple bounds; in what follows $C$ is a constant that only depends on the (uniform) Alfhors-regularity bounds on $\om_i$ and $\om_\infty$. First, we see that by the $L^2$ bound for $\xi - \chi_{B(P,R)}$ we have
\begin{equation}\label{eq101.eq}
\begin{split}
\left| \int \varphi\cdot \nu_\infty \, d\sigma_\infty - \vec{N}\cdot \int_{B(P,R)}\nu_\infty \, d\sigma_\infty\right| &\le \sigma_\infty(B(P,2R))^{\frac{1}{2}} \lVert \xi - \chi_{B(P,R)}\rVert_{L^2(d\sigma_\infty)} 
\\& < C\epsilon \sigma_\infty(B(P,R))
\end{split}
\end{equation}
where we used the Ahlfors-regularity of $\sigma_\infty$ in the last inequality. 

Also, since $\left|\nu_\infty(P) - \fint_{B(P,R)}\nu_\infty(Q) d\sigma_\infty(Q)\right| < \epsilon$ we may conclude from \eqref{eq101.eq}
\begin{equation}\label{eq102.eq}
\begin{split}
\left| \int \varphi\cdot \nu_i \, d\sigma_i - \sigma_\infty(B(P,R)) \vec{N}\cdot \nu_\infty \right| & < (C + 1)\epsilon \sigma_\infty(B(P,R))
\\& \le C\epsilon \sigma_\infty(B(P,R)).
\end{split}
\end{equation}
Next, by Lemma \ref{KTlemma} (in particular the fact that $\chi_{\om_i}\rightarrow \chi_{\om_\infty}$ in $L^1_{\mathrm{loc}}$), we have that $\nu_i d\sigma_i \rightharpoonup \nu_\infty \sigma_\infty$; to see this we need only note that $[d\chi_{\Omega_i}] = \nu_i d\sigma_i$ and $[d\chi_{\Omega_\infty}] = \nu_\infty d\sigma_\infty$ since the measure theoretic boundaries of the sets $\om_i$ and $\om_\infty$ coincide with the topological boundaries of $\om_i$ and $\om_\infty$ respectively (see \cite[Chapter 5]{evansandgariepy}). Therefore, there exists an $i_0$ (which depends on $\varphi$) such that for all $i \ge i_0$
\begin{equation}\label{eq103.eq}
\left| \int \varphi\cdot \nu_i \, d\sigma_i - \int \varphi\cdot \nu_\infty \, d\sigma_\infty\right| < \epsilon \sigma_\infty(B(P,R)).
\end{equation}
We claim that, perhaps after adjusting $i_0$, we have for all $i \ge i_0$,
\begin{equation}\label{eq104.eq}
\int \varphi \cdot \nu_i d\sigma_i + \epsilon \sigma_i(B(P,R)) > \vec{N} \cdot \int_{B(P,R)} \nu_i \, d\sigma_i,
\end{equation}
or, equivalently,
\begin{equation}\label{eq105.eq}
\int_{B(P,2R)\setminus B(P,R)} \xi \vec{N}\cdot \nu_i \, d\sigma_i + \epsilon \sigma_i(B(P,R)) > 0,
\end{equation}
where we used that $\chi_{B(P,R)} \le \xi \le \chi_{B(P,2R)}$.
Indeed, we have a stronger statement, there exists $i_0$ such that for all $i \ge i_0$,
\begin{equation}\label{eq106.eq}
\sigma_i(\{z \in B(P,2R): \nu_i(z) \cdot \vec{N} < 0\}) < \frac{\epsilon}{2}\sigma_i(B(P,R)).
\end{equation}
To prove \eqref{eq106.eq} we have that $|\nu_i| < 1$ $\sigma_i$ a.e. so that $1 - \vec{N}\cdot \nu_i \ge 0$, $\sigma_i$-a.e.,  and by \eqref{alllimeq.eq} if $i$ is sufficiently large
\begin{equation}\label{eq107.eq}
\left|\fint_{B(P,2R)} \vec{N}\cdot \nu_i \, d\sigma_i - 1 \right| < c \epsilon,
\end{equation}
where $c$ is a small constant depending on the uniform Ahlfors regularity to be chosen momentarily. Since $1 - \vec{N}\cdot \nu_i \ge 0$, $\sigma_i$-a.e., \eqref{eq107.eq} implies 
$$\left|\fint_{B(P,2R)} 1 -  \vec{N}\cdot\nu_i \, d\sigma_i \right| < c \epsilon,$$
so that by Chebyshev's inequality 
$$\sigma_i(\{z \in B(P,2R): 1 - \vec{N} \cdot \nu_i(z) > 1\}) \le \sigma_i(B(P,2R))c\epsilon < \frac{\epsilon}{2} \sigma_i(B(P,R)),$$
by choice of $c$. This clearly implies \eqref{eq106.eq}. Now we make the observation that
\begin{equation*}
\begin{split}
\limsup_{i \to \infty} \sigma_i(B(P,R)) &= \limsup_{i \to \infty} \sigma_i(B(P,R)) \vec{N}\cdot\fint_{B(P,R)} \nu_i \, d\sigma_i
\\&= \limsup_{i \to \infty} \vec{N} \cdot \int_{B(P,R)} \nu_i \, d\sigma_i. 
\end{split}
\end{equation*}
Thus, combining \eqref{eq104.eq}, \eqref{eq103.eq} and \eqref{eq101.eq}
we obtain 
\begin{equation}\label{eq108.eq} 
\begin{split}
(1 - \epsilon) \limsup_{i \to \infty} \sigma_i(B(P,R)) &< \limsup_{i \to \infty} \int \varphi\cdot \nu_i \, d\sigma_i
\\ &\le \int \varphi \cdot \nu_\infty \, d\sigma_\infty + \epsilon\sigma_\infty(B(P,R))
\\ & \le \sigma_\infty(B(P,R))(\vec{N} \cdot \nu_\infty(P) + C\epsilon).
\end{split}
\end{equation}
By the lower semicontinuity of the total variation of $BV$ functions we have 
$$\sigma_\infty(B(P,R)) \le \liminf_{i \to \infty} \sigma_i(B(P,R)) \le  \limsup_{i \to \infty} \sigma_i(B(P,R))$$
so that \eqref{eq108.eq} implies
$$\frac{\vec{N} \cdot \nu_\infty(P) + C \epsilon}{1 - \epsilon} > 1$$
for any $\epsilon \in (0,1)$. It follows that $\nu_\infty(P) = \vec{N}$.
\end{proof}

\begin{remark} It is interesting to note (particularly in the setting of Corollary \ref{nuVMOthrm.thrm}) that using the same analysis above one can show $\sigma_i \rightharpoonup \sigma_\infty$. Recall that we had already shown $\nu_\infty \equiv \vec{N}$, taking $f \in C_c(\rn)$ with $\supp f \subset B(0,R)$ we show
\begin{equation}\label{rmksurfmsrconv.eq}
\lim_{i \to \infty} \left| \int (f \nu_i)\cdot \nu_i \, d\sigma_i - \int (f \vec{N})\cdot\vec{N} \, d\sigma_\infty \right| = 0
\end{equation}
Since $f$ is arbitrary and $|\nu_i| = 1$, $\sigma_i$-a.e., this will show $\sigma_i \rightharpoonup \sigma_\infty$. Using the triangle inequality we have that the left hand side of \eqref{rmksurfmsrconv.eq} is bounded by 
\begin{equation}
\begin{split}
&\limsup_{i \to \infty} \left| \int (f \nu_i)\cdot \nu_i \, d\sigma_i - \int (f \vec{N})\cdot \nu_i \, d\sigma_i \right|
\\ & \qquad + \limsup_{i \to \infty} \left| \int (f \vec{N})\cdot \nu_i \, d\sigma_i  - \int (f \vec{N})\cdot\vec{N} \, d\sigma_\infty \right|
\\ & \quad = \limsup_{i \to \infty} \left| \int (f \nu_i)\cdot \nu_i \, d\sigma_i - \int (f \vec{N})\cdot \nu_i \, d\sigma_i \right|
\\ & \quad \le \limsup_{i \to \infty} \sigma_i(B(0,R)\lVert f \rVert_\infty \fint_{B(0,R)} (1 - \vec{N}\cdot \nu_i) \, d\sigma_i  = 0
\end{split}
\end{equation}
where the equality comes from the weak convergence $\nu_i d\sigma_i \rightharpoonup \nu_\infty d\sigma_\infty$ and in the last line we used \eqref{alllimeq.eq}, the uniform Ahlfors regularity property and that  $1 - \vec{N}\cdot \nu_i \ge 0$, $\sigma_i$-a.e.
\end{remark}

\begin{remark}\label{scalingconstt1.eq} We remark that in  Theorem \ref{nuBMOthrm.thrm} the constant $c_\delta$ does not depend on $R_0$ in the two sided corkscrew condition while $r_\delta$ does.  For this reason we may often reduce (by scaling) to the case $R_0 = 1$. We also notice that control on the oscillation of $\nu_i$ centered around $Q_i \in \pom_i$ gives control on the flatness of $\pom_i$ around $Q_i$. We illustrate these observations with the following two refinements/Corollaries of Theorem \ref{nuBMOthrm.thrm}.
\end{remark}

%In the case of unbounded domains, it will be convenient to ``localize" our hypothesis.

\begin{corollary}[A Pointwise Version of Theorem \ref{nuBMOthrm.thrm}]\label{cor:locally}Suppose $\om \subset \ree$ is an open set satisfying the $(M,R_0)$-two-sided corkscrew condition whose boundary is Ahlfors regular.
Define for $Q \in \pom$ and $R > 0$
\begin{equation}\label{eqn:localizedosc}\lVert \nu \rVert_{\ast}(Q, R) := \sup_{0 < s < R} \left( \fint_{B(Q,s)} \left| \nu(z) - \fint_{B(Q,s)} \nu(y) \, d\sigma(y) \right|^2 \, d \sigma(z) \right)^\frac{1}{2}.\end{equation}
Then for every $\delta > 0$ there exists $c_\delta$ and $r_\delta$ depending only on $\delta$, the Ahlfors regularity constant and the constants in the two-sided corkscrew condition such that if  
$$\lVert \nu \rVert_{\ast}(Q, 1 )< c_\delta,$$ 
then 
$$\sup_{r \in (0, r_\delta)} \Theta(Q,r) < \delta.$$
\end{corollary}

\begin{proof}
The proof is exactly the same as Theorem \ref{nuBMOthrm.thrm}. 
Proceeding in a way as to obtain a contradiction, we suppose there exist domains $\om^{(i)}$  satisfying the hypotheses above and $Q_i \in \om^{(i)}$ with $\lVert \tilde{\nu}_i \rVert_{\ast}(Q_i, 1 ) <1/i$  but $\Theta(Q_i, r_i) \ge \delta$ for some $r_i < 1/i$ (here $\tilde{\nu}_i$ is the unit outer normal to $\om^{(i)}$ and the norm $\|\cdot \|_*$ is taken with respect to $\tilde{\sigma}_i$ to surface measure to $\om^{(i)}$). Again, without loss of generality we may assume $Q_i \equiv 0$. We then proceed exactly as in Theorem \ref{nuBMOthrm.thrm} noting that the analysis hinges on the following estimate for all $R > 0$, 
$$\limsup_{i \to \infty} \fint_{B(0,Rr_i)} \left| \tilde{\nu}_i(z) - \fint_{B(0,Rr_i)} \tilde{\nu}_i(x) \, d\tilde{\sigma}(x) \right| \, d\tilde{\sigma}_i(z) = 0.$$
The expression indexed by $i$ in the limit superior is less than $1/i$ provided that $Rr_i < 1$, which always occurs for $i$ large enough.

\end{proof}

Finally, we have that the unit normal in $\mathrm{VMO}$ implies vanishing Reifenberg flatness. 
%\simon{\Bl The old corollary had many things wrong with it. ADR and Corkscrews at a point are not enough. You would likely need a topological assumption and you would definitely need the corkscrew and ADR conditions to hold in a neighborhood of a point. Even so, it would definitely require writing a proof.}
\begin{corollary}\label{nuVMOthrm.thrm}Suppose $\om \subset \ree$ is a open set satisfying the $(M,R_0)$-two-sided corkscrew condition whose boundary is Ahlfors regular. If $\nu \in VMO_{loc}(d\sigma)$
where $\sigma$ is the surface measure for $\om$ and $\nu$ is the outer unit normal to $\om$, then for every compact set $K \subset \ree$
$$\lim_{r \to 0} \sup_{Q \in \pom \cap K} \Theta(Q,r) = 0,$$
where $\Theta(Q,r)$ is associated to the set $\pom$.
\end{corollary}

\begin{proof}
Fix a compact set $K$ and $\delta > 0$. Using the definition of $VMO_{loc}$ we let $s > 0$ be such that 
$$\sup_{Q \in K}\lVert \nu \rVert_{\ast}(Q,s )< c_\delta,$$
where $c_\delta$ is as in Corollary \ref{cor:locally}.
Fix $Q \in K$. Set $\widetilde{\om} := \frac{\om - Q}{s}$ and let $\tilde{\nu}$ and $\tilde{\sigma}$ be the unit normal and surface measure to $\widetilde{\om}$ respectively. It follows from a change of variables that $\lVert \tilde{\nu} \rVert_{\ast}(Q,1 )< c_\delta$ and hence by Corollary \ref{cor:locally} we have 
$$\sup_{r \in (0, r_\delta)} \theta_{\widetilde{\om}} (0, r) < \delta,$$
where $\theta_{\widetilde{\om}}$ is associated to $\partial\widetilde{\om}$. Dilating and translating we have that 
$$\sup_{r \in (0, sr_\delta)} \theta (Q, r) < \delta.$$
Since $sr_\delta$ is uniform over $Q\in K$ we have shown $\limsup_{r \to 0} \sup_{Q \in \pom \cap K} \Theta(Q,r) < \delta$. As $\delta > 0$ is arbitrarily small, the corollary follows.
\end{proof}

\subsection{Applications to Chord-Arc Domains with Small Constant}\label{sec:applicationstocad}

Let us recall the definition of a Chord Arc Domain with Small Constant (or $\delta$-Chord Arc Domain) introduced by Kenig and Toro (see, e.g. Definitions 1.10 and 1.11 in \cite{kenigtoroannsci}). 

\begin{definition}[$\delta-$Chord Arc Domain and Vanishing Chord Arc Domain]\label{def:cadsc} We say a domain, $\om \subset \ree$, is a $\delta-$chord arc domain (or chord arc domain with small constant) if $\om$ is a $\delta-$Reifenberg flat chord arc domain and for each compact set $K \subset \ree$ there exists a $R >0$ such that
\begin{equation*}
\sup_{Q \in \pom \cap K} \lVert \nu \rVert_{*}(Q,R) < \delta,
\end{equation*}
where $\nu$ is the unit outer normal to the boundary and the norm $\|\cdot\|_*$ (recall \eqref{eqn:localizedosc}) is with respect to $\sigma$, the surface measure of $\pom$.

We say a domain $\om$ is a chord arc domain with vanishing constant if it is a chord arc domain with small constant and for each compact set $K \subset \ree$ 
\begin{equation}\label{VMOnuVCADdef.eq}
\lim_{r \to 0}\sup_{Q \in \pom \cap K} \lVert \nu \rVert_*(Q,r)  = 0.
\end{equation}
That is to say, $\nu \in \mathrm{VMO}_{loc}(d\sigma)$. 
\end{definition}

 A consequence of our Theorem \ref{nuBMOthrm.thrm} and Corollary \ref{nuVMOthrm.thrm} is that the {\it a priori} assumption of Reifenberg flatness (assuming the presence of corkscrews) in Definition \ref{def:cadsc} is redundant. We should remark that ``quantitatively" this is not precisely true; Theorem \ref{nuBMOthrm.thrm} does not guarantee that a chord arc domain with $\|\nu\|_{\mathrm{BMO}} < \delta$ is a $\delta$-chord arc domain, only that it is a $\varepsilon$-chord arc domain for some $\varepsilon = \varepsilon(\delta) > 0$. However, for most applications the precise value of $\delta > 0$ is not particularly important (see, e.g. Corollary 5.2 in \cite{kenigtoroduke}). 

However, before we can continue we make a quick remark as to the complications which arise when we consider unbounded chord arc domains. 

\begin{remark}\label{unboundedspellstrouble}
When working with unbounded domains, there is an issue that $\delta$-Reifenberg flatness of $\partial \Omega$ and $\nu \in \mathrm{VMO}_{loc}$ are conditions which hold at scales that are merely uniform over compacta, whereas the NTA conditions are required to hold at uniform scales throughout the domain (see Example \ref{baddomain.ex} below for a potentially problematic domain). Therefore, in order to ensure the unbounded domains in the results below are   NTA, we must assume, {\it a priori}, some global flatness (or smallness of $\mathrm{BMO}$ norm). We note that the need for global control is also why $(\delta, R)$-Reifenberg flatness is included in the definition of unbounded Reifenberg flat domains.

%On the other hand,  the definition of chord arc domains merely required that the domain be {\it locally} NTA (i.e. have NTA conditions which hold at scales that are uniform over compact sets), then Corollary \ref{bddvmocor.cor} would need no additional {\it a priori} global assumptions. 
\end{remark}

The following example illustrates the issues discussed in Remark \ref{unboundedspellstrouble}. The example is a modification of an example provided to the first author by Steve Hofmann while working on \cite{bortzhofmann}.

\begin{example}\label{baddomain.ex} For $n \ge 2$, there exists a domain $\om \subset \mathbb{R}^n$ with the following properties:
\begin{enumerate}
\item $\pom$ is AR,
\item $\om_{ext} =\interior(\om^c)$ is non-empty and consists of one connected component,
\item $\om$ satisfies the two-sided corkscrew condition,
\item the outer normal to $\om$, $\nu$, satisfies $\nu \in VMO_{loc}$,
\item $\pom$ is vanishing Reifenberg flat and
\item $\om$ fails the $(C,R)$-Harnack chain condition for any $C, R > 0$.
\end{enumerate}
\end{example}
We describe the domain now and leave it to the reader to verify properties (1) - (6).
Given $s \in (0,1)$ we let $\Gamma_s \subset \mathbb{R}^2$ be the curve pictured below.
\squeezeup 
\begin{center}
\begin{figure}[h!]%[labelsep= none]
\includegraphics{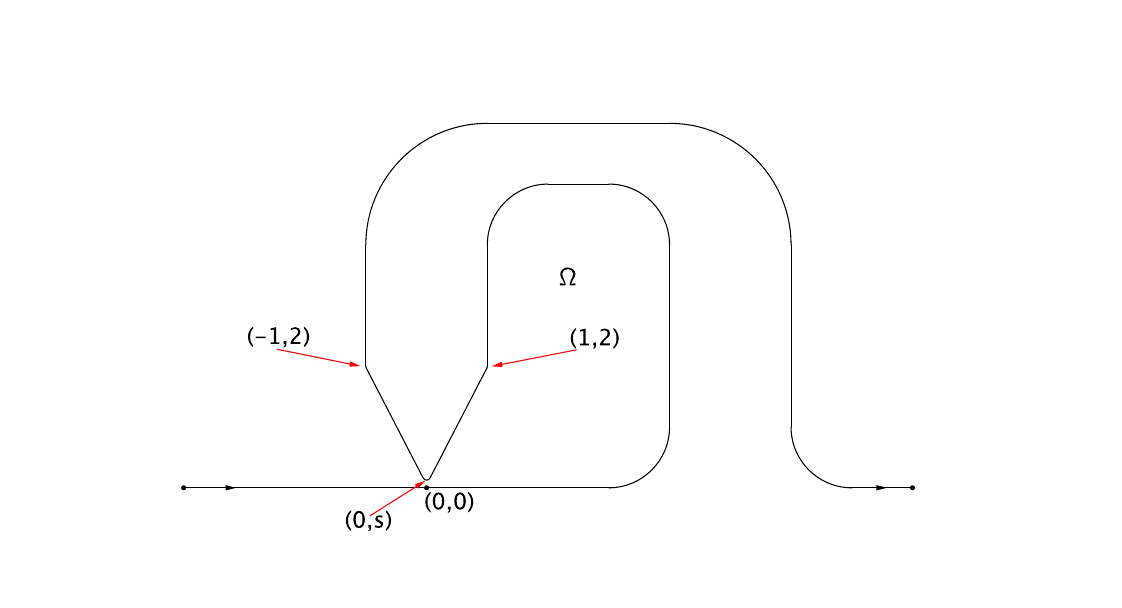}
\end{figure}
\end{center}

Here the only portion of $\Gamma_s$ that depends on $s$ is the the curve, $\C_s$, connecting $(-1,2)$ to $(1,2)$, the rest of the curve is made up of (horizontal and vertical) line segments and quarter-circle arcs. We give a brief description of $\C_s$ so the reader may verify properties (1) - (5). $\C_s$ is a ``smoothed" version of the curve given by the line segment from $(1,2)$ to $(0,s)$ followed by the line segment from $(0,s)$ to $(-1,2)$. We construct $\C_s$ from circular arcs from $\partial B((1-\tfrac{s}{2}, 2), \tfrac{s}{2}),\  \partial B((0,\tfrac{3s}{2}), \tfrac{s}{2})$ and $\partial B((-1 + \tfrac{s}{2}, 2),\tfrac{s}{2})$, and two line segments, one segment parallel to the the line through $(1-\tfrac{s}{2}, 2)$ and $(0,\tfrac{3s}{2})$, and the other parallel to the line through $(-1 + \tfrac{s}{2}, 2)$ and $(0,\tfrac{3s}{2})$ in such a way that $\Gamma_s$ is $C^1$. To ensure corkscrew points for $\om$, is it important to see that $\C_s$ stays above $y = |x|$. 

Now, we construct $\om$. We use the convention that $\C \cup \C'$ means $\C$ followed by $\C'$, where we attach the beginning point of $\C'$ to the endpoint of $\C$ and allow our curves to be defined up to translation. Let $\Gamma$ be the curve of infinite length given by 
$$\Gamma :=  \ \ldots \cup \Gamma_{1/4} \cup \Gamma_{1/3} \cup \Gamma_{1/2} \cup \Gamma_{1/3} \cup\Gamma_{1/4} \cup \dots.$$
Clearly, $\mathbb{R}^2\setminus \Gamma$ consists of two non-empty connected components and we let $\om$ be as pictured above. Conditions (1) and (2) are satisfied readily. Some elementary but tedious calculations show that (3) holds. To see that (4) and (5) hold, note that at every point, up to rotation, $\om$ is locally given by the region above the graph of a $C^1$ function and in the definition of vanishing Reifenberg flat and $VMO_{loc}$, we are checking that a condition is satisfied on all compact sets (but not uniformly). The failure of the $(C,R)$-Harnack chain condition follows from the pinching that occurs in $\Gamma_s$ near $(0,0)$ as $s$ tends to zero (and $\C_s$ stays above $y = |x|$). We can use this example in $\mathbb{R}^2$ to obtain a similar example in $\mathbb{R}^n$ for $n \ge 3$ by taking $\om_2 \times \mathbb{R}^{n-2}$, where $\om_2$ is the example constructed here (in $\mathbb{R}^2$).

\begin{corollary}\label{BMOimpflatnessdomains.cor} Let $\delta \in (0, \delta_n]$. Suppose that $\om \subset \ree$ is a  domain satisfying the $(M,R_0)$-two-sided corkscrew condition whose boundary is Ahlfors regular. If $\om$ is unbounded additionally assume that $\ree \setminus \pom$ consists of two nonempty connected components. There exists $c_\delta < \delta$ depending only on $M, \delta$ and the AR constant such that if $\sup_{Q \in \pom} \lVert \nu \rVert_{*}(Q,R_1) < c_\delta$ for some $R_1$, then $\om$ is a $\delta-$chord arc domain with constants depending on $M, \delta, R_0, R_1$ and the AR constant.
\end{corollary}

\begin{proof}
Let $c_\delta$ be as in Corollary \ref{cor:locally} (we may assume $c_\delta < \delta$) and $Q \in \pom$ be arbitrary. Set $R := \min\{R_0,R_1, 1\}$. Define $\widetilde{\om} := \frac{\om - Q}{R}$, then $\widetilde{\om}$ satisfies the $(M,1)$ corkscrew condition and $\partial\widetilde{\om}$ is AR with the same AR constant as $\om$. Moreover, $\tilde{\nu}$, the unit outer normal to $\widetilde{\om}$ satisfies $\lVert \tilde{\nu} \rVert_*(0,1) < c_\delta$. Applying Corollary \ref{cor:locally}, we obtain the existence of $r_\delta$ such that 
$\sup_{r \in (0,r_\delta)} \Theta_{\widetilde{\om}}(0, r) < \delta < \delta_n$. Scaling and translating back to $\om$ and noting that $Q$ was arbitrary we obtain 
$$\sup_{r \in (0,Rr_\delta)}\sup_{Q \in \pom} \Theta_{\widetilde{\om}}(Q, r) < \delta < \delta_n.$$ Thus, $\om$ is $(\delta, Rr_\delta)$-Reifenberg flat. 
\end{proof}

Corollary \ref{nuVMOthrm.thrm} implies a similar result for vanishing chord arc domains. Note that for unbounded domains this does not improve Definition \ref{def:cadsc} substantively.

%Theorem \ref{nuVMOthrm.thrm} shows that, in the setting of bounded domains, {\it a priori} flatness for vanishing chord arc domains is redundant.

\begin{corollary}\label{bddvmocor.cor} Let $\om$ be a domain, whose boundary is Ahlfors regular and which satisfies the $(M,R_0)$-two-sided corkscrew condition. If $\om$ is unbounded, we assume that $\ree \setminus \pom$ consists of two nonempty connected components. Let $c_{\delta_n}$ be as in Corollary \ref{cor:locally}. Then any of the following imply that $\om$ is a Vanishing chord arc domain.
\begin{enumerate}
\item $\nu \in VMO_{loc}$ and $\pom$ is bounded.
\item $\lVert \nu \rVert_{BMO} < c_{\delta_n}$ and $\nu \in VMO_{loc}$.
\item $\sup_{Q \in \pom} \lVert \nu \rVert_*(Q,R_1) < c_{\delta_n}$ for some $R_1 > 0$ and $\nu \in VMO_{loc}$.
\end{enumerate}
\end{corollary}
\begin{proof} By definition $(1) \implies (3)$ and $(2) \implies (3)$. So, we may assume that (3) holds. Corollary \ref{BMOimpflatnessdomains.cor} then shows that $\om$ is a $\delta_n$-chord arc domain and clearly for any compact set $K$ we have
$$\lim_{r \to 0} \sup_{Q \in \pom \cap K} \lVert \nu \rVert_*(Q,r) = 0.$$
The result follows by Corollary \ref{nuVMOthrm.thrm}.
\end{proof}

\section{Application to a Two-Phase Problem For Harmonic Measure}

In the sequel, we will assume $n \ge 2$. In this section, we apply the results of the previous section to a two-phase free boundary problem for harmonic measure, originally studied by Kenig-Toro in \cite{kenigtorotwophase}.  Our approach however, will be a quantified version of that of the first author with Hofmann \cite{bortzhofmann}. In particular, as in \cite{bortzhofmann} we avoid any {\it a priori} assumption on topology. %wo reasons; first is that Kenig-Toro use pseudo-blowups, and it is unclear how to make this argument quantitative. The second reason is that the strategy of \cite{bortzhofmann} does not require the {\it a priori} assumption that $\om$ is NTA in the case that the pole is finite, which is crucial in the pseudo-blowup argument. 

Let us introduce the necessary definitions and notation so that we may state the free boundary problem. We will assume that $\Omega^{\pm}$ are uniformly rectifiable (UR) domains (first introduced by Hofmann, Mitrea and Taylor in \cite{hofmannmitreataylor}). To properly define UR domains, we must first recall what it means for a set to be uniformly rectifiable.  We note that the following ``definition" is actually two (quite deep) theorems and that a proper introduction to uniformly rectifiable sets would first give a more geometric characterization (see, e.g. \cite{DavidandSemmes1,DavidandSemmes2}). However, for our purposes, the following characterization of UR sets is most useful: 

\begin{definition}[UR]\label{defur} (aka {\it uniformly rectifiable}).
Let $E\subset \ree$ be an $n$-dimensional Ahlfors regular (hence closed) set with surface measure $\sigma$. Then $E$ is uniformly rectifiable (UR) if and only if the Riesz transform operator, $\mathcal R$ is $L^2$ bounded with respect to surface measure, in the sense that
 \begin{equation}\label{eqrtbound}
 \sup_{\eps>0} \|\mathcal{R}_\eps f\|_{L^2(E,\sigma)} \le C\rVert f \rVert_{L^2(E,\sigma)}\,,
 \end{equation}
(see Definition \ref{riesztransdef} for a definition of $\mathcal R$ and $\mathcal{R}_\eps$). That uniform rectifiability (defined in a geometric sense) implies the Riesz transforms are bounded is due to David and Semmes, \cite{DavidandSemmes1}. The converse is due to \cite{mattilamelnikovverdera} when $n=1$, and \cite{nazarovtolsavolberg} in general. 

The constant $C > 0$ in \eqref{eqrtbound} and the constant implicit in the Ahlfors regularity determine the ``UR character" of $E$. Below, we will denote the dependence of a constant, $K$, on the UR character by $K(UR)$. 
\end{definition}

%\begin{definition}\label{defurchar} ({\bf ``UR character"}).   Given a UR set $E\subset \ree$, its ``UR character"
%is just the pair of constants $(\theta,M_0)$ involved in the definition of uniform rectifiability,
%along with the ADR constant; or equivalently,
%the quantitative bounds involved in any particular characterization of uniform rectifiability. When we say a constant %depends on UR, written $C(UR)$ we mean that the constant depends on the UR character of the underlying set. 
%\end{definition}

We can now define a UR domain:
\begin{definition}[UR domain, see \cite{hofmannmitreataylor}]\label{URdom}
We will say that a domain $\om$ is a UR domain if $\pom$ is UR,
and if the measure theoretic boundary
$\partial_*\om$ (see \cite[Chapter 5]{evansandgariepy})
satisfies $\sigma(\pom \setminus  \partial_*\om)=0$.
\end{definition}

Note, in particular, that if an Ahlfors regular domain satisfies the two-sided corkscrew conditions then it is a UR domain (that a domain which satisfies the two-sided corkscrew condition with Ahlfors regular boundary has a UR boundary is a result of David and Jerison \cite{davidandjerison}). Additionally, one should note that if $\om \subset \ree$ is a set of locally finite perimeter then the measure theoretic boundary and the reduced boundary differ by a set of $H^n$ measure zero, so it then follows that the measure theoretic boundary has full measure if and only if the reduced boundary has full measure (see \cite[Section 5.8]{evansandgariepy}).  

For any domain $\om$ with ADR boundary and surface measure $\sigma = \mathcal H^n |_{\pom}$, we adopt the notation for $r > 0$
$$\lVert f \rVert_*(r) := \sup_{Q \in \pom} \lVert f \rVert_*(Q,r),$$
where $\lVert f \rVert_*(Q,r)$ is as in \eqref{eqn:localizedosc}. 

We can now state our theorem (we have two theorems, one for finite and the other for infinite pole):

\begin{theorem}[{Quantified version of \cite[Theorem 1.1]{bortzhofmann}}]\label{FBthrm.thrm}
Let $\om^+ \subset \ree$ and $\om^-= \ree\setminus \overline{\om^+}$ be connected UR domains with common boundary $\pom^+ = \pom^-$ and $\diam(\pom^+) < \infty$. Let $X^+ \in \om^+$ and $X^- \in \om^-$ be such that $k^+ = \frac{d\hm^{X^+}}{d\sigma}$ and $k^- = \frac{d\hm^{X^-}}{d\sigma}$ exist. Let $\delta > 0$ and $r_0 \in (0, \diam(\pom^+))$. There exists $$\eta = 
\eta (\delta,n, UR,r_0, \delta(X^+),\delta(X^-), \diam(\pom^+)) > 0$$  such that if 
$$\lVert \log k^+ \rVert_*(r_0),\lVert \log k^- \rVert_*(r_0) < \eta,$$
then $\om^+$ and $\om^-$ are $(\delta)$-chord arc domains.
\end{theorem}

We state the next theorem without proof (see Remark \ref{onefiniteoneinfinite}). Note we assume that the domains are NTA domains from the outset, as it is not clear how to even define the Poisson kernel with pole at infinity without this assumption.

%\simon{\Rd I am not sure about the dependence on $r_0$... The lemmas in KT03 are not so clear. For instance, Theorem 2.1 in KT03 depends on the VMO character, while Corollary 2.4 claims to be independent of this. Evidently, just glancing at the proof of Corollary 2.4, this is not true... I think we should play it safe and keep the dependence of $r_0$.}
\begin{theorem}[{Quantified version of \cite[Theorem 1.2 (2)]{bortzhofmann}}]\label{UnboundedFBthrm}
Let $\om^+ \subset \ree$ and $\om^-= \ree\setminus \overline{\om^+}$ be $(M,\infty)$-chord arc domains with common boundary $\pom^+ = \pom^-$ and $\diam(\pom^+) = \infty$. Then for any $\delta > 0$ there exists $\eta = \eta(n, AR, M, \delta) > 0$ such that  if 
$$\lVert \log k^+ \rVert_{BMO(d\sigma)}, \lVert \log k^+ \rVert_{BMO(d\sigma)} < \eta$$
then $\om^+$ and $\om^-$ are $\delta$-chord arc domains. Here $k^+$ and $k^-$ are the Poisson kernels with pole at infinity for the domains $\om^+$ and $\om^-$ respectively.
\end{theorem}

A few remarks are in order.
%\simon{\Rd I think we should add a third item in this remark and say that the reason we chose these theorems is to mirror the work in KT-Duke '93. You have been a bit better with the history, could you do this?}
\begin{remark}\label{onefiniteoneinfinite}
1) We omit the proof of Theorem \ref{UnboundedFBthrm} because it follows from  \cite[Theorem 1.2 (2)]{bortzhofmann} and Corollary \ref{cor:locally} in much the same way as Theorem \ref{FBthrm.thrm}. In fact, the hypothesis that $\Omega^\pm$ are $(M,\infty)$-chord arc domains allows one to immediately apply the ``CFMS" estimates \cite{cfms}, which makes the proof of Theorem \ref{UnboundedFBthrm} simpler than that of Theorem \ref{FBthrm.thrm} (see \cite{bortzhofmann} for more details).

%The techniques in \cite[Theorem 1.2 (2)]{bortzhofmann} applied {\it verbatim} and Corollary \ref{cor:locally} yield Theorem \ref{UnboundedFBthrm} directly (see the proof of Theorem \ref{FBthrm.thrm} below). The assumptions in Theorem \ref{UnboundedFBthrm} allow one to use the so-called CFMS  \cite{cfms} estimates for the Poisson kernel at infinity  \cite{kenigtoroannals}. This makes the proof of Theorem \ref{UnboundedFBthrm} much simpler than Theorem \ref{FBthrm.thrm}, see \cite{bortzhofmann}.

2) To simplify matters, we will prove Theorem \ref{FBthrm.thrm} when $\mathrm{diam}(\pom^+) = 1.$ For this reason, we will state many of the lemmas below for $\diam(\pom^+) = 1$. A simple scaling argument recovers the general case.
\smallskip

3)  We note that \cite{kenigtoroduke} refers to $(\delta, R)$-chord arc domains which are simply $\delta$-chord arc domains where the flatness and the oscillation of the unit normal are controlled up to scale $R > 0$ globally. In Theorem \ref{FBthrm.thrm} above (and in Definition \ref{def:cadsc}) we follow the lead of \cite{kenigtoroannsci} and suppress the  scale. However, our methods allow us to keep track of the scale at which the flatness and oscillation are controlled and we try to make that clear in the proof below. 
\end{remark}

To prove Theorem \ref{FBthrm.thrm}, we hope to employ Corollary \ref{cor:locally}. Therefore, we must produce the two-sided corkscrew condition.  The following lemma, while a simple consequence of observations in \cite{azzammourgogloutolsavolberg} and \cite{azzamhofmannmartellnystromtoro}, may be of independent interest. 

\begin{lemma}\label{uranddoublingimpliescorkscrews}
Let $\Omega^{+} \subset \mathbb R^{n+1}$ and $\om^- = \ree \setminus \overline{\om^+}$ be connected UR domains in $\mathbb R^{n+1}$ with common boundary $\partial \Omega^+ = \partial \Omega^-$. Let $C > 0$ and suppose $X^{\pm} \in \Omega^{\pm}$ are such that 
\begin{equation}\label{eqn:doublingcondition} \omega^{X^{\pm}}(\Delta(Q,r)) \leq C\omega^{X^{\pm}}(\Delta(Q,r/2))\end{equation} 
for all $Q\in \partial \Omega^{+}$ and $0< r < \tfrac{1}{4}\min\{\delta(X^+), \delta(X^-), \diam(\pom^+)\}$. That is, $\omega^{X^{\pm}}$ satisfy a doubling condition.

Then there exists $M, R_0 > 0$, depending on $n, UR, \delta(X^+), \delta(X^-)$ and the constant $C$ in \eqref{eqn:doublingcondition} such that $\Omega^{\pm}$ satisfy the two-sided $(M, R_0)$-corkscrew condition. 
\end{lemma}

\begin{proof}
Set $\pom:= \pom^+$. First, we claim that for all $\epsilon > 0$, $Q \in \pom$ and $r \in (0, \diam(\pom^+))$ there exists $\widetilde{Q} \in \pom$ and $\tilde{r} \approx_\epsilon r$ with $B(\tilde{Q},\tilde{r}) \subset B(Q,r)$ and an affine $n$-plane $V(\widetilde{Q},\tilde{r})$ satisfying
\begin{equation}\label{eq3.9.eq}
D[V(\widetilde{Q},\tilde{r})\cap B(\widetilde{Q},\tilde{r}); \partial \Omega \cap B(\widetilde{Q},\tilde{r})] < \varepsilon \tilde{r}.
\end{equation}
Indeed, the bilateral weak geometric lemma (BWGL) \cite[Theorem 2.4]{DavidandSemmes2} guarantees the existence of $\widetilde{Q}$, $\tilde{r}$ and $V(\widetilde{Q},\tilde{r})$. The BWGL states, if $\pom$ is UR, the failure of \eqref{eq3.9.eq} is quantified by a Carleson packing condition. Then a pigeon-hole argument yields  $\widetilde{Q}$, $\tilde{r}$ and $V(\widetilde{Q},\tilde{r})$, see the discussion following \cite[Lemma 4.1]{azzamhofmannmartellnystromtoro}.

Now, let $Q \in \pom$, $0 < r < \tfrac{1}{16}\min\{\delta(X^+), \delta(X^-), \diam(\pom^+)\}$ and let $\widetilde{Q} \in \pom$, $\tilde{r}$ and $V(\widetilde{Q},\tilde{r})$ be as above with $0 < \epsilon \ll 1$ to be chosen.
We define two sets 
$$B^\pm(\widetilde{Q},\tilde{r}) := \{z \in B(\widetilde{Q},\tilde{r})\mid \pm \left\langle z- \tilde{Q}, \hat{n}_V\right\rangle > \varepsilon \tilde{r}\},$$
where $\hat{n}_V$ is a perpendicular to $V(\widetilde{Q},\tilde{r})$. It is well known (see \cite{davidandjerison}) that if $\pom$ is Ahlfors regular, then $\pom^c$ satisfies the corkscrew condition. Thus, for all $0 < \rho < \diam(\pom^+)$ and $P \in \pom^+$ there is $Y_{P,\rho} \in B(P,\rho) \cap \pom^c$ with $\dist(Y_{P,\rho}, \pom) > \rho/M$, where $M$ depends on AR. Without loss of generality (symmetry of hypothesis), and by choice of $\epsilon$ small and choice of the sign of $\hat{n}_V$, we may assume $A^+ \equiv Y_{\widetilde{Q}, \tilde{r}} \in B^+(\widetilde{Q}, \tilde{r}) \cap \om^+$. Note that $\pom \cap B^\pm(\widetilde{Q},\tilde{r}) = \emptyset$. In particular, the convexity of $B^+(\widetilde{Q},\tilde{r})$ implies that there cannot be points from both $\Omega^+$ and $\Omega^-$ inside  $B^+(\widetilde{Q},\tilde{r})$ (similarly for $B^-(\widetilde{Q}, \tilde{r})$).  We claim $B^-(\widetilde{Q},\tilde{r}) \cap \om^- \neq \emptyset$. If this claim is true, then the above argument shows $B^-(\widetilde{Q},\tilde{r}) \subset \Omega^-$ and we set $A^- = \tilde{Q} - 2\varepsilon \tilde{r}\hat{n}_V\in \Omega^-\cap B(\tilde{Q}, \tilde{r})$.

To this end, we appeal to \cite[Lemma 3.3]{azzammourgogloutolsavolberg}. In our setting, this lemma yields the following: if $\mathcal H^{n+1}(B(\widetilde{Q},\tilde{r})\cap \Omega^+) \geq \kappa \tilde{r}^{n+1}$, then the doubling of $\omega^{X^-}$ implies that $\mathcal H^{n+1}(B(\widetilde{Q},\tilde{r})\cap \Omega^-) \geq \tilde{\kappa}\tilde{r}^{n+1}$, where $\tilde{\kappa} > 0$ depends on $C$ and $\kappa > 0$. Since $B(Y_{\widetilde{Q},\tilde{r}}, \tilde{r}/M) \subset B(\widetilde{Q}, \tilde{r}) \cap \om^+$, this lemma yields $B(\widetilde{Q},\tilde{r}) \cap \om^-$ must intersect $B^+(\widetilde{Q}, \tilde{r})$ or $B^-(\widetilde{Q},\tilde{r})$ non-trivially for all $\epsilon$ sufficiently small. Since $B^+(\widetilde{Q}, \tilde{r}) \subset \om^+$ it must be the case that $B^-(\widetilde{Q}, \tilde{r}) \subset \om^-$. Having verified our claim, the points $A^+, A^-$ (defined above) suffice as corkscrew points for $Q$ at scale $r$ (we can change location and scale as $B(\tilde{Q}, \tilde{r}) \subset B(Q,r)$ and $\tilde{r} \approx r$). 
\end{proof}

\begin{remark}We quickly remark that \cite[Lemma 2.1]{azzammourgogloutolsavolberg} is used to prove Lemma 3.3 in \cite{azzammourgogloutolsavolberg}. While \cite[Lemma 2.1]{azzammourgogloutolsavolberg} is stated for bounded domains, the lemma always holds for open sets with Ahlfors regular boundary (regardless of boundedness). 
\end{remark}

Our main tool in the proof of Theorem \ref{FBthrm.thrm} is the single layer potential, we recall its definition now:

\begin{definition}[Riesz transforms and the single layer potential]\label{riesztransdef}
Let $E\subset \ree$ be an $n$-dimensional AR (hence closed) set with surface measure $\sigma$.
We define the (vector valued) Riesz kernel as
\begin{equation}\label{RieszKern}
\mathcal{K}(x) = \tilde{c}_n\frac{x}{|x|^{n+1}}
\end{equation}
where $\tilde{c}_n$ is chosen so that $\mathcal{K}$ is the gradient of fundamental solution to the Laplacian. 
For a Borel measurable function $f$, we then define the Riesz transform 
\begin{equation}\label{RieszXform}
\mathcal{R}f(X) := \mathcal{K} \ast (f\sigma)(X) = \int_{E} \mathcal{K}(X-y)f(y)\, d\sigma(y) \quad X \in \ree\,,
\end{equation} 
as well as the truncated Riesz transforms
$$\mathcal{R}_\eps f(X):= \int_{E\,\cap\,\{|X-y|>\eps\}} \mathcal{K}(X-y)f(y)\, d\sigma(y)\,,\qquad  \eps>0\,.$$
We define $\mathcal{S}$ the single layer potential for the Laplacian relative to $E$ to be
\begin{equation}
\mathcal{S}f(X): = \int_{E} \mathcal{E}(X -y) f(y) \, d\sigma(y),
\end{equation}
where $\mathcal{E}(X) = c_n|X|^{1-n}$ is the (positive) fundamental solution to the 
Laplacian in $\ree$. Notice that $\nabla \mathcal{S}f(X) = \mathcal{R}f(X)$ for $X \not\in E$.
\end{definition}

%The singular layer potential is useful in that it gives solutions to the Neumann problem. However, in order to make sense of boundary data in a rough domain we need to introduce the concept of non-tangential regions: 

\begin{definition}[Nontangential approach region and maximal function]\label{NTapp} Fix $\alpha > 0$ and let $\om$ be a domain,
then for $x \in \pom$ we define the nontangential approach region (or ``cone")
\begin{equation}\label{NTapp1}
\Gamma(x) = \Gamma_\alpha(x) = \{Y \in \om : |Y - x| < (1 + \alpha)\delta(Y)\}. 
\end{equation}
We also define the nontangential maximal function for $u: \om \to \re$
\begin{equation}\label{NTapp2}
\mathcal{N}u(x) = \mathcal{N}_\alpha u(x) = \sup_{Y \in \Gamma_\alpha(x)}|u(Y)|, \quad x \in \pom.
\end{equation}
We make the convention that $\mathcal{N}u(x) = 0$ when $\Gamma_\alpha(x)=\emptyset$.
\end{definition}

The relationship between the two definitions above is made clear in the following two lemmas:

\begin{lemma}[\cite{hofmannmitreataylor}, \cite{DavidandSemmes1}]
For all $p \in (1,\infty)$ we have
\begin{equation}\label{eq15}
\lVert \mathcal{N}(\nabla \mathcal{S} f)\rVert_{L^p(d\sigma)}  \le C \lVert f \rVert_{L^p(d\sigma)},
\end{equation}
where $C$ depends on the UR character of $\pom$, dimension, $p$, and the aperture of the cones defining
$\mathcal{N}$. 
\end{lemma}
Estimate \eqref{eq15} is essentially proved in \cite{DavidandSemmes1}; bounds for the non-tangential maximal
function of $\nabla \mathcal{S} f$ follow from uniform bounds for the truncated singular integrals, plus a standard 
Cotlar Lemma argument;  the details may be found in \cite[Proposition 3.20]{hofmannmitreataylor}.

In addition, we have the following result proved in \cite{hofmannmitreataylor}.
\begin{lemma}[\cite{hofmannmitreataylor} Proposition 3.30]\label{hmtlemma} 
If $\om$ is a UR domain (recall Definition \ref{URdom}), %whose measure theoretic boundary has full measure 
then for  a.e. $x \in \pom$, and for all $f \in L^p(d\sigma)$, $1<p<\infty$,
\begin{equation}\label{eq16}
\lim_{\substack{Z \to x \\ Z \in \Gamma^-(x)}}\nabla \mathcal{S} f (Z) = -\frac{1}{2}\nu(x)f(x) + Tf(x)\,,
\end{equation}
and 
\begin{equation}\label{eq30}
\lim_{\substack{Z \to x \\ Z \in \Gamma^+(x)}}\nabla \mathcal{S} f (Z) = \frac{1}{2}\nu(x)f(x) + Tf(x)\,.
\end{equation}
where $T$ is a principal value singular integral operator, $\Gamma^+(x)$ is the cone at 
$x$ relative to $\om$, $\Gamma^-(x)$ is the cone at $x$ relative to ${\Omega_{\rm ext}}$, and $\nu$ is the 
outer normal to $\om$. 
\end{lemma}

\begin{remark} As in \cite{bortzhofmann}, we have taken our fundamental solution to be positive, so 
for that reason there are some changes in sign in both \eqref{eq16} and \eqref{eq30} 
as compared to the formulation in \cite{hofmannmitreataylor}. \end{remark}

%Now we fix $x \in \pom$ and $A^+$ and $A^-$ corkscrew points for $\om$ and $\om_{ext}$ respectively with the property that $|A^+ - x|, |A^- - x| < R_0$ and $B(A^+,R_0/M) \subset \om$ and $B(A^-, R_0, M) \subset \om_{ext}$. We then denote $k^+$ the Poisson kernel for $\om$ with pole at $A^+$ and $k^-$ the Poisson kernel for $\om_{ext}$ with pole at $A^-$. Finally note that if we set $\alpha := M$ we have that $\Gamma^+_\alpha(y) \cap B(y,r) \neq \emptyset$ and $\Gamma^-_\alpha(y) \cap B(y,r) \neq \emptyset$ for all $y \in \pom$ and $0 < r < R_0$. Note that $A^+$ and $A^-$ serve as corkscrew points relative to all $Q \in \pom$ at twice the scale.

Now we recall several lemmas from \cite{kenigtoroannsci} and \cite{bortzhofmann}, most of which require no modification (we indicate the necessary adjustments when this is not the case). The following lemma is a direct result of the John-Nirenberg inequality, which continues to hold for Ahlfors regular sets (see \cite[Corollary 2.19, p. 409]{GR}). 

 \begin{lemma}\label{RH4est.lem} Let $\om$ be a UR domain with $\diam(\pom) = 1$ and let $0 < r_0 < 1$. Let $f \ge 0$. There exists $\kappa_1= \kappa_1(n,AR)$ and $C = C(n,AR, r_0)$ such that if 
$\lVert \log f \rVert_*(r_0) < \kappa_1$ then 
\begin{equation}\label{eq2}
\left(\fint_{\Delta} f^4 \, d\sigma\right)^{\frac{1}{4}} \le C \fint_\Delta f\, d\sigma,
\end{equation}
for all $\Delta = \Delta(x,r) = B(x,r) \cap \pom$ with $x \in \pom$, $r \in (0,1)$. That is, $f \in RH_4(d\sigma)$. 
\end{lemma}
To prove the Lemma, as in the case of the Euclidean space the John-Nirenberg inequality allows us to show that if $b = \log f$, $e^{\eta b} \in A_2 \subset A_\infty$ where $A_2$ is the Muckenhaupt class where $\eta= \eta(n, AR)$ and $\eta \to \infty$ as $\lVert b \rVert_*(r_0) \to 0$. Since $f \in RH_p$ if and only $f^p \in A_\infty$ (see \cite{C-UrN}) we then obtain that $f \in RH_4$, provided $\eta > 4$. 
\begin{remark}
The reverse H\"older estimate \eqref{eq2} 
(i.e., the $A_\infty$ property) yields an exponential reverse Jensen inequality, so that
for any $\Delta$ as in  \eqref{eq2},
\begin{equation}\label{eq6}
e^{\,\fint_{\Delta} \log f \, d \sigma} \approx \fint_{\Delta} f \, d \sigma = \frac{1}{\sigma(\Delta)} \int_{\Delta} f \, d \sigma.
\end{equation}
See \cite[Theorem 2.15, p. 405]{GR}.
\end{remark}

\begin{lemma}{\cite[Lemma 1.16]{bortzhofmann}}\label{Lemma 2}
Let $\om$ be a UR domain with $\diam(\pom) = 1$. Let $f \ge 0$ with $\lVert \log f \rVert_*(r_0) < \kappa_1$, where $\kappa_1$ is as in Lemma \ref{RH4est.lem}. For $\Delta^* := \Delta(x,s) =  B(x,s) \cap \pom$ with $0 \le s \le r_0$ and $x \in \pom$, set
$$a_{x,s} := e^{\fint_{\Delta^*} \log f \, d\sigma}.$$
There exists $C = C(n, AR, r_0)$ such that 
\begin{equation}\label{eq7}
\left(\fint_{\Delta^*} \left|1 - \frac{f}{a} \right|^2 \, d \sigma \right)^{1/2} \le C \left(\lVert \log f \rVert_{*}(r_0)\right)^{1/8} ,
\end{equation}

\end{lemma}

\begin{proof} Following the proof of \cite[Lemma 1.16]{bortzhofmann} {\it verbatim}, we replace $\epsilon$ with $\lVert \log f \rVert_{*}(r_0)$ and set $p = 2$. We also note that in the last estimate in the proof of \cite[Lemma 1.16]{bortzhofmann} it is required that $f \in RH_{2p} = RH_4$, at this point we apply Lemma \ref{RH4est.lem}. 
\end{proof}

If we place additional smallness assumptions on the $\|\cdot\|_*$-norm of $\log f$ we obtain the additional comparability estimate below. The proof of this lemma is identical to \cite[Corollary 2.4]{kenigtoroannsci} , appealing to the fact that we may place $f \in RH_p$ for any $p> 1$ provided we force $\log f$ to have small enough $\lVert \cdot \rVert_*$-norm.

\begin{lemma}[\cite{kenigtoroannsci} Corollary 2.4]\label{Lemma 1} Suppose that
Let $\om$ be a UR domain with $\diam(\pom) = 1$ and let $0 < r_0 < 1$. There exists $\kappa_2= \kappa_2(n,AR)$ and $C = C(n,AR, r_0)$ such that if $\lVert f \rVert_*(r_0) < \kappa_2$ and $\mu(A) = \int_A f \, d\sigma$ then
\begin{equation}\label{eq4}
\begin{split}
C^{-1}\left(\frac{\sigma(E)}{\sigma(\Delta^\star)} \right)^{1+(1/2n)} \le \frac{\mu(E)}{\mu(\Delta^\star)} \le 
C \left(\frac{\sigma(E)}{\sigma(\Delta^\star)} \right)^{1-(1/2n)},
\end{split}
\end{equation}
for all surface balls $\Delta^\star$ and $E \subset \Delta^\star$.
\end{lemma}

%We now state and prove our free boundary result. Again, we will make the simplification that $R_0 = \diam(\pom) = 1$ since we may always scale our domain to have this property without changing the $BMO$ norms. As remarked previously we expect to have the scale at which flatness to occur to depend on the diameter of the boundary.

%\begin{theorem}\label{FBthrm.thrm} Let $\delta \in (0, \delta_n)$ and suppose that $\om$ is a domain with $\diam(\pom) = 1$, satisfying the $(M,1)-$two-sided corkscrew condition and $R^n \setminus \pom$ consists of two non-empty connected components $\om$ and $\om_{ext} = \interior(\om^c)$. Fix $x \in \pom$ and let $A^+,A^-, k^+, k^-$ be as above. Then there exists $\eta = \eta(AR, M, n)$ such that if 
%$$\lVert \log k^+ \rVert_{BMO(d\sigma)}, \lVert \log k^- \rVert_{BMO(d\sigma)} < \eta,$$
%then $\om$ is a $\delta-$Reifenberg flat chord arc domain.
%\end{theorem}

\begin{proof}[Proof of Theorem \ref{FBthrm.thrm}] By scaling we may assume $\diam(\pom^+) = 1$ and $0 < r_0 < 1$. Our proof will follow the main scheme of \cite{bortzhofmann} at first, in fact, our situation is slightly simpler.
Suppose first that 
$$\lVert \log k^+ \rVert_*(r_0),\lVert \log k^- \rVert_*(r_0) \le \min\{\kappa_1, \kappa_2\},$$
where $\kappa_1$ and $\kappa_2$  are the constants from Lemma \ref{RH4est.lem} and Lemma \ref{Lemma 2} respectively. By Lemma \ref{RH4est.lem}, $\hm^{X^+},\hm^{X^-} \in A_\infty$ and hence $\hm^{X^+}$ and $\hm^{X^-}$ are doubling measures. From this fact and Lemma \ref{uranddoublingimpliescorkscrews}, it follows that $\om^+$ and $\om^-$ satisfy the two-sided $(M,R_0)$-corkscrew condition with uniform constants depending on $n, UR, r_0,$ $\delta(X^+), \delta(X^-)$. 

Let $\lVert \log k^+\rVert_*(r_0), \lVert \log k^+\rVert_*(r_0) < \eta$ where $\eta$ is small to be chosen. 
 We first assume that $\eta \ll \min\{r_0, \delta(X^+), \delta(X^-)\}< 1$. Fix $Q \in \pom$ and $r \in (0,\eta)$ and set $\Delta = \Delta(Q,r)$. For $y, z \in \Delta$, let $y^*$, $z^*$ denote arbitrary points in $\Gamma^-(y) \cap B(y,r/2)$ and in $\Gamma^-(z) \cap B(z,r/2)$ respectively. Set $\Delta^* = \Delta(Q, \eta^{-1/(8n)}r)$, where we have chosen $\eta$ in such a way that the radius of $\Delta^*$ is (significantly) smaller than $\min\{r_0, \delta(X^+), \delta(X^-)\}$. As in \cite{bortzhofmann}, our immediate goal is to show
\begin{equation}\label{eq17}
\left( \fint_\Delta \left| \nabla \mathcal{S} 1_{\Delta^*}(z^\ast) - \fint_\Delta \nabla \mathcal{S} 1_{\Delta^*}(y^*) \, d\sigma(y) \right|^2 \, d \sigma(z) \right)^\frac{1}{2} \le C \eta^{\gamma}, 
\end{equation}
where $\gamma := 1/(8n)$ and $C= C(\delta, n, UR, r_0, \delta(X^+), \delta(X^-))$.

Set $k := k^+$, $\displaystyle a: = a_{x,\eta^{-1/(8n)}r} = e^{\,\fint_{\Delta^*} \log k \, d \sigma}$ and write
\begin{equation}\label{eq18}
1_{\Delta^*} = \left[\left(1 - \frac{k}{a}\right)1_{\Delta^*}\right] + \left[\frac{k}{a}\right] - \left[\left(\frac{k}{a}\right)1_{(\Delta*)^c}\right].
\end{equation}
Using \eqref{eq18} we have that the left hand side of \eqref{eq17} is bounded by the sum of three terms $\RNum{1}, \RNum{2}$ and $\RNum{3}$ where
\begin{equation}\label{eq19}
\RNum{1} = \left( \fint_\Delta \left| \nabla \mathcal{S} \left[\left(1 - \frac{k}{a}\right)1_{\Delta^*}\right] (z^\ast) - \fint_\Delta \nabla \mathcal{S} \left[\left(1 - \frac{k}{a}\right)1_{\Delta^*}\right](y^*) \, d\sigma(y) \right|^2 \, d \sigma(z) \right)^\frac{1}{2},
\end{equation}
\begin{equation}\label{eq20}
\RNum{2} = \left( \fint_\Delta \left| \nabla \mathcal{S} \left[\frac{k}{a}\right] (z^\ast) - \fint_\Delta \nabla \mathcal{S} \left[\frac{k}{a}\right](y^*) \, d\sigma(y) \right|^2 \, d \sigma(z) \right)^\frac{1}{2},
\end{equation}
and
\begin{equation}\label{eq21}
\RNum{3} = \left( \fint_\Delta \left| \nabla \mathcal{S} \left[\left(\frac{k}{a}\right)1_{(\Delta*)^c}\right] (z^\ast) - \fint_\Delta \nabla \mathcal{S} \left[\left(\frac{k}{a}\right)1_{(\Delta*)^c}\right](y^*) \, d\sigma(y) \right|^2 \, d \sigma(z) \right)^\frac{1}{2}.
\end{equation}

 We begin by estimating $\RNum{1}$.  By \eqref{eq15} and Lemma \ref{Lemma 2}, we have
\begin{multline}\label{eq23}
\RNum{1} \le  2 \left(\fint_\Delta \left| \mathcal{N} 
\left( \nabla \mathcal{S} \left[\left(1 - \frac{k}{a}\right)
1_{\Delta^*}\right]\right) \right|^2 \, d \sigma \right)^{\frac{1}{2}} \\*
\lesssim \, \left(\frac{\sigma(\Delta^*)}{\sigma(\Delta)}\right)^{1/2}
\left(\fint_{\Delta^\ast} 
\left| 
1 - \frac{k}{a} 
\right|^2 \, d \sigma \right)^{\frac{1}{2}}\,
\lesssim \, \eta^{\frac{-1}{16}} \eta^{\frac{1}{8}} \,\lesssim\, \eta^{\frac{1}{16}}.
\end{multline}
Now for $\RNum{2}$, we recall that $k = k^{X^+}$ is the Poisson kernel for $\om$ with pole at $X^+$.  Moreover,
$\mathcal{E}(\cdot - z^*)$ and $\mathcal{E}(\cdot - y^*)$ are 
harmonic in $\om$ since $z^*, y^* \in {\Omega_{\rm ext}}$, and decay to 0 at infinity, and are therefore equal 
to their respective Poisson integrals in $\om$.   Consequently,
\begin{equation}\label{eq24}
%\begin{split}
\RNum{2} % &=\frac{1}{a}  \left( \fint_\Delta \left| \nabla \mathcal{E}({X_1} - z^*)  
% - \fint_\Delta \nabla \mathcal{E}({X_1} - y^*) \, d\sigma(y) \right|^2 \, d \sigma(z) 
% \right)^\frac{1}{2} \\ &
\le \frac{1}{a}  \left( \fint_\Delta  \fint_\Delta \left| \nabla \mathcal{E}({X^+} - z^*)  - \nabla \mathcal{E}({X^+} - y^*) \, d\sigma(y) \right|^2 \, d \sigma(z) 
\right)^\frac{1}{2}\,.
%\end{split}
\end{equation}
We now apply Lemma \ref{Lemma 1},
with $\Delta^\star = \Delta(Q, 1) = \pom$ and $E=\Delta^*$, to deduce that
\begin{equation}\label{eq2.8a}
\frac{\hm(\pom)}{\hm(\Delta^*)} = \frac{\hm(\Delta_0^\star)}{\hm(\Delta^*)} \lesssim 
\left(\frac{1}{\eta^{\frac{-1}{8n}}r}\right)^{n+1/2},
\end{equation}
where $\hm= k^+ \, d\sigma$.
Note that, since $y^*,z^* \in B(x,2r)$, 
\begin{equation*}%\label{eq37}
\left| \nabla \mathcal{E}({X^+} - z^*)  - \nabla \mathcal{E}({X^+} - y^*)   \right| 
\lesssim \frac{r}{\delta(X^+)^{n+1}} \approx r,
\end{equation*}
where we remind the reader that the implicit constants may depend on $\delta(X^+)$.
Then continuing \eqref{eq24}, we have, using \eqref{eq6} and %Lemma \ref{bourgainsLemma} 
\eqref{eq2.8a}  %and Lemma \ref{Lemma 1} (with $\beta = \frac{1}{2n}$),
\begin{multline}\label{eq38}
\begin{split}
\RNum{2}  &\lesssim \frac{1}{a} r 
\approx \frac{\sigma(\Delta^*)}{\hm(\Delta^*)} r  =\frac{\sigma(\Delta^*)}{\hm(\pom)}\frac{\hm(\pom)}{\hm(\Delta^*)} r 
\\ & \lesssim \,(\eta^{\frac{-1}{8n}}r)^n\left(\frac{1}{\eta^{\frac{-1}{8n}}r} \right)^{n+\frac{1}{2}}r  \lesssim 
\eta^{\frac{1}{16n}}r^{\frac{1}{2}}  \lesssim  \eta^{\frac{1}{8n}},
\end{split}
\end{multline}
where we have used the estimate $\hm(\pom) \gtrsim 1$ with implicit constants depending on $n$ AR and $\delta(X^+)$. To see the estimate $\hm(\pom) \gtrsim 1$, we appeal to Bourgain's estimate \cite{Bourgain}. Note that this holds for pole near the boundary, but by using the touching ball for the point $X^+$ and the Harnack inequality it holds with pole at $X^+$, albeit with constants depending on $\delta(X^+)$. Recall the touching ball for a point $X \in \om^+$ is the ball $B(X,R)$ with $R = \dist(X,\pom^+)$.

For $\RNum{3}$, we use basic Calder\'{o}n-Zygmund type estimates as follows.  Let 
$$\Delta_j' := \Delta(Q, 2^jr)\,,\qquad A'_j := \Delta_j' \setminus \Delta_{j-1}'\,,$$ 
so that
\begin{multline}\label{eq25}
\RNum{3} = \\ %&= \left( \fint_\Delta \left| \nabla \mathcal{S} \left[\left(\frac{k}{a}\right)
%1_{(\Delta*)^c}\right] (z^\ast) - \fint_\Delta \nabla \mathcal{S} 
%\left[\left(\frac{k}{a}\right)1_{(\Delta*)^c}\right](y^*) \, d\sigma(y) \right|^2 \, d \sigma(z) \right)^\frac{1}{2} \\
 \left( \fint_\Delta \left| \fint_\Delta 
 \left(\nabla \mathcal{S} \left[\left(\frac{k}{a}\right)
 1_{(\Delta*)^c}\right] (z^\ast) -  \nabla \mathcal{S} \left[\left(\frac{k}{a}\right)
 1_{(\Delta*)^c}\right](y^*)\right) \, d\sigma(y) \right|^2 \, d \sigma(z) \right)^\frac{1}{2} \\
 = \left( \fint_\Delta \left| \fint_\Delta \int_{\pom\setminus \Delta^*} \right[\nabla 
\mathcal{E}(z^* - w) -\nabla \mathcal{E}(y^*-w)\left]\frac{k(w)}{a} \, d\sigma(w)  
\, d\sigma(y) \right|^2 \, d \sigma(z) \right)^\frac{1}{2}  \\
\le \sum_{\{j\mid2^j \ge \eta^{-\frac{1}{8n}}\}} \left( \fint_\Delta \left[\fint_\Delta   
 \int_{A'_j} \left|\nabla \mathcal{E}(z^* - w) -\nabla \mathcal{E}(y^*-w)\right|\frac{k(w)}{a} \, d\sigma(w) 
  \, d\sigma(y) \right]^2 \, d \sigma(z) \right)^\frac{1}{2} \\
 \lesssim  \sum_{\{j\mid 2^j \ge \eta^{-\frac{1}{8n}}\}} \left( \fint_\Delta \left[ \fint_\Delta  \int_{A'_j} 
\frac{r}{(2^jr)^{n+1}}\frac{k(w)}{a} \, d\sigma(w)  \, d\sigma(y) \right]^2 \, d \sigma(z) \right)^\frac{1}{2},
% \le  \sum_{j: 2^j \ge M}  \frac{1}{2^j}
%\left( \fint_\Delta \left[ \fint_\Delta  \fint_{\Delta_j'} \frac{k(w)}{a} \, 
%d\sigma(w)  \, d\sigma(y) \right]^2 \, d \sigma(z) \right)^\frac{1}{2} \\
% \le \, \sum_{j: 2^j \ge M}  \frac{1}{2^k} \frac{1}{a}\frac{\hm(\Delta_j')}{\sigma(\Delta_j')}\,
%\,\lesssim \,\frac{1}{M} \,\lesssim \,\epsilon^{\frac{1}{8n}}\,,
\end{multline}
where we understand that the sums are finite and terminate for $2^j r  \ge \mathrm{diam}(\partial \Omega)= 1$.
%$$J_1:= \{j: 2^j \geq M \,\, {\rm and}\,\, B(x,2^jr) \subset B_0^\star\}\,,$$ and
%$$J_2:= \{j: 2^j \geq M \,\, {\rm and}\, \,B(x,2^jr) \,\, {\rm meets}\, \,\ree \setminus B_0^\star\}\,.$$
We now apply Lemma \ref{Lemma 1}, with $\Delta^\star = \Delta_j'$ and $E = \Delta^*$, to obtain
\begin{equation}\label{eq2.12}
\frac{\hm(\Delta_j')}{\hm(\Delta^*)} \lesssim  \left(\frac{2^{j}}{\eta^{-\frac{1}{8n}}}\right)^{n+ 1/2} 
\end{equation}
We then have
\begin{multline*}
III \,\lesssim\,\,
 \sum_{\{j\mid 2^j \ge \eta^{-\frac{1}{8n}}\}} \, \frac{1}{2^j}
\left( \fint_\Delta \left[ \fint_\Delta  \fint_{\Delta_j'} \frac{k(w)}{a} \, 
d\sigma(w)  \, d\sigma(y) \right]^2 \, d \sigma(z) \right)^\frac{1}{2} \\
 \lesssim \, \sum_{\{j\mid 2^j \ge \eta^{-\frac{1}{8n}}\}} \, \frac{1}{2^j} \frac{1}{a}\frac{\hm(\Delta_j')}{\sigma(\Delta_j')}\,
\approx \, \sum_{\{j\mid 2^j \ge \eta^{-\frac{1}{8n}}\}} \, \frac{1}{2^j} \frac{\sigma(\Delta^*)}{\sigma(\Delta_j')} \frac{\hm(\Delta_j')}{\hm(\Delta^*)}\\
\lesssim\, \sum_{\{j\mid 2^j \ge \eta^{-\frac{1}{8n}}\}} 2^{-j}  \left(\frac{\eta^{-\frac{1}{8n}}}{2^{j}}\right)^n  \left(\frac{2^{j}}{\eta^{-\frac{1}{8n}}}\right)^{n+1/2} 
\\ \,\lesssim \eta^{\frac{1}{16n}}\sum_{\{j\mid 2^j \ge \eta^{-\frac{1}{8n}}\}} 2^{-j/2}  
 \,\lesssim \,\eta^{\frac{1}{8n}}\,,
\end{multline*}
where in the second line we have used  \eqref{eq6} and
in the second to last line the AR property and \eqref{eq2.12}.
Combining the estimates for $I$, $II$, $III$, we obtain \eqref{eq17}.

Setting
\begin{equation*}%\label{eq27}
\nt^- f(x) := \lim_{\substack{Z \to x \\ Z \in \Gamma^-(x)}} \nabla \mathcal{S}f(Z) \, ,
\end{equation*}
since the limit exists for a.e. $x\in \pom$ (see Lemma \ref{hmtlemma}),
we may now use \eqref{eq15},  \eqref{eq17}, and 
dominated convergence to obtain  
\begin{equation}\label{eq28}
\left( \fint_\Delta \left| \nt^- 1_{\Delta^*}(z) - \fint_\Delta \nt^- 1_{\Delta^*}(y) \, d\sigma(y) \right|^2 \, d \sigma(z) \right)^\frac{1}{2} \le C \eta^\gamma. 
\end{equation}

In addition, since $\lVert \log k^{-} \rVert_*(r_0) < \eta$, the same analysis shows that \eqref{eq28} holds for $\nt^- 1_{\Delta^*}$ replaced with
\begin{equation}\label{eq29}
\nt^+ 1_{\Delta^*} : = \lim_{\substack{Z \to x \\ Z \in \Gamma^+(x)}}\nabla\mathcal{S}1_{\Delta^*}(Z).
\end{equation}

By \eqref{eq16} and \eqref{eq30}
\begin{equation}\label{eq31}
\nu(x) 1_{\Delta^*}(x) = \lim_{\substack{Z \to x \\ Z \in \Gamma^+(x)}}\nabla\mathcal{S}1_{\Delta^*}(Z) - \lim_{\substack{Z \to x \\ Z \in \Gamma^-(x)}}\nabla\mathcal{S}1_{\Delta^*}(Z)\,.
\end{equation}
%%so that one has
%\begin{equation}\label{eq32}
%\left( \fint_\Delta \left| \nu(z) 1_{\Delta^*}(z) - \fint_\Delta \nu(y) 1_{\Delta^*}(y) \, 
%d\sigma(y) \right|^2 \, d \sigma(z) \right)^\frac{1}{2} \le C\epsilon^\gamma,
%\end{equation}
%that is, 
Thus, since $\Delta \subset \Delta^*$,
by \eqref{eq28} and its analogue for $\mathcal{S}^+$, we obtain for all $\Delta = \Delta(x,r)$ with $x \in \pom^+$ and $r \in (0,\eta)$
\begin{equation}\label{eq33}
\left( \fint_\Delta \left| \nu(z) - \fint_\Delta \nu(y) \, d\sigma(y) \right|^2 \, d \sigma(z) \right)^\frac{1}{2} \le C\eta^\gamma.
\end{equation}

Now, we depart from \cite{bortzhofmann} and apply a variant of our compactness argument in Theorem  \ref{nuBMOthrm.thrm}. Notice that, writing \eqref{eq33} compactly, we have
$$\lVert \nu \rVert_{*}(Q, \eta) < C \eta^\gamma,$$ 
where $\lVert \nu \rVert_{*}(Q, \eta)$ is as in Corollary \ref{cor:locally}. We remind the reader that $\om^+$ satisfies the two-sided $(M,R_0)$-corkscrew condition with constants independent of $\eta$.
Choosing $\eta$ such that $C\eta^\gamma < c_\delta$ we apply the ideas of Corollary \ref{cor:locally} to obtain that $\pom$ is $(\delta, \eta r_\delta)$-Reifenberg flat (and thus a $(\delta, \eta r_\delta)$-chord arc domain). 
\end{proof}

We close with a final modification to  Theorems 1.1 and 1.2 in \cite{bortzhofmann}. In the presence of an additional hypothesis to ensure the existence of corkscrews, then in the setting of Theorems 1.1 and 1.2 in \cite{bortzhofmann}, we may also conclude that $\pom$ is vanishing Reifenberg flat. This follows immediately from Corollary \ref{nuVMOthrm.thrm}. In fact, using Lemma \ref{uranddoublingimpliescorkscrews}, an additional hypothesis is not required except for the situation that $\pom$ has infinite diameter and both poles are finite.   Indeed, we had this application in mind when we began working on this paper.

\section*{Acknowledgments}
The authors would like to thank Steve Hofmann and Tatiana Toro for their helpful discussions. 

\bibliography{BEVMO}{}
\bibliographystyle{amsalpha}

\end{document}